\documentclass[hidelinks, 10pt]{article}

\usepackage[UKenglish]{babel}
\usepackage[utf8]{inputenc}
\usepackage[T1]{fontenc}
\usepackage{lmodern}
\usepackage{amsmath}
\usepackage{amsthm}
\usepackage{amsfonts}
\usepackage{amssymb}
\usepackage{mathtools}
\usepackage{eucal}  
\usepackage{dsfont} 
\usepackage{authblk} 

\usepackage[dvipsnames]{xcolor}
\usepackage{hyperref}
\usepackage{fullpage}
\usepackage{microtype}
\usepackage{newpxmath,newpxtext}
\usepackage{booktabs}
\usepackage{enumitem} 
\usepackage{comment}

\usepackage{adjustbox}
\usepackage{quiver}

\tikzset{node distance=2cm, auto}
\tikzcdset{row sep/normal=2.7em,column sep/normal=3.5em}

\allowdisplaybreaks

\setlist[description]{font=\normalfont\bfseries\:\!}

\renewcommand\labelenumi{(\alph{enumi})}
\renewcommand\theenumi\labelenumi

\hypersetup{
    colorlinks=true,
    linkcolor=gray,
    filecolor=gray,      
    urlcolor=gray,
    citecolor=gray
    }
\numberwithin{equation}{section}

\newtheorem{theorem}{Theorem}[section]
\newtheorem{proposition}[theorem]{Proposition}
\newtheorem{lemma}[theorem]{Lemma}
\newtheorem{corollary}[theorem]{Corollary}

\theoremstyle{definition}
\newtheorem{definition}[theorem]{Definition}
\newtheorem{example}[theorem]{Example}
\newtheorem{remark}[theorem]{Remark}

\newenvironment{acknowledgements}{\vskip .5cm\noindent\small\textbf{Acknowledgements}.\quad\noindent}{}

\newcommand{\CategoryFont}[1]{\mathsf{\uppercase{#1}}}
\newcommand{\FunctorFont}[1]{\mathsf{#1}}

\renewcommand{\|}{\:|\:}

\renewcommand{\bar}[1]{\mkern 1.5mu\overline{\mkern-1.5mu#1\mkern-1.5mu}\mkern 1.5mu}

\newcommand{\x}{\otimes}

\renewcommand{\.}{\cdot}
\renewcommand{\,}{,\dots,}

\renewcommand{\=}{\:\colon\!=}
\renewcommand{\epsilon}{\varepsilon}
\renewcommand{\o}{\circ}
\renewcommand{\b}{\bullet}
\newcommand{\<}{\langle}
\renewcommand{\>}{\rangle}
\newcommand{\op}{\mathsf{op}}
\newcommand{\co}{\mathsf{co}}

\renewcommand{\*}{\ast}
\newcommand{\nto}{\nrightarrow}

\newcommand{\Inc}{\mathsf{Inc}}

\newcommand{\N}{\mathbb{N}}

\newcommand{\id}{\mathsf{id}}

\newcommand{\CC}{{\mathbf{K}}}
\newcommand{\DD}{{\mathbf{J}}}
\newcommand{\X}{\mathbb{X}}
\newcommand{\Y}{\mathbb{Y}}
\newcommand{\End}{\CategoryFont{End}}

\newcommand{\Cat}{\CategoryFont{Cat}}
\newcommand{\Mnd}{\CategoryFont{Mnd}}

\newcommand{\Fib}{\CategoryFont{Fib}}
\newcommand{\Indx}{\CategoryFont{Indx}}
\newcommand{\RestrCat}{\CategoryFont{RCat}}
\newcommand{\sRestrCat}{\CategoryFont{sRCat}}

\newcommand{\TngCat}{\CategoryFont{TngCat}}

\newcommand{\Weil}{\CategoryFont{Weil}_1}

\newcommand{\T}{\mathrm{T}}
\newcommand{\TT}{\mathbb{T}}
\newcommand{\Tng}{\CategoryFont{Tng}}

\newcommand{\Leung}{\mathfrak{L}}
\newcommand{\q}{\mathbf{q}}
\newcommand{\p}{\mathbf{p}}
\newcommand{\TngFib}{\CategoryFont{TngFib}}
\newcommand{\TngIndx}{\CategoryFont{TngIndx}}
\newcommand{\TngRestrCat}{\CategoryFont{TngRCat}}
\newcommand{\TngsRestrCat}{\CategoryFont{TngSRCat}}

\newcommand{\U}{\CategoryFont{U}}
\newcommand{\M}{\mathscr{M}}

\newcommand{\Univ}[1]{\mathcal{\uppercase{#1}}}

\renewcommand{\H}{\mathsf{H}}

\newcommand{\Split}{\FunctorFont{Split}}

\newcommand{\Insert}{\CategoryFont{Insert}}
\newcommand{\Equif}{\CategoryFont{Equif}}

\newcommand{\IND}{\mathscr{I}}
\newcommand{\I}{\mathbb{I}}
\newcommand{\E}{\mathbb{E}}

\newcommand{\VF}{\CategoryFont{VF}}

\title{The formal theory of tangentads\\ {\Large PART I}}
\date{}
\author{\uppercase{Marcello Lanfranchi}}
\affil{\normalsize\textit{Macquarie University, School of Mathematical and Physical Sciences}}

\makeatother
\makeatletter
\begin{document}

\maketitle

\begin{abstract}\noindent
Tangent categories offer a categorical context for differential geometry, by categorifying geometric notions like the tangent bundle functor, vector fields, Euclidean spaces, vector bundles, connections, etc. In the last decade, the theory has been extended in new directions, providing concepts such as tangent monads, tangent fibrations, tangent restriction categories, reverse tangent categories and many more. It is natural to wonder how these new flavours interact with the geometric constructions offered by the theory. How does a tangent monad or a tangent fibration lift to the tangent category of vector fields of a tangent category? What is the correct notion of vector bundles for a tangent restriction category? We answer these questions by adopting the formal approach of tangentads. Introduced in our previous work, tangentads provide a unifying context for capturing the different flavours of the theory and for extending constructions like the Grothendieck construction or the equivalence between split restriction categories and $\M$-categories, to the tangent-categorical context. In this paper, we construct the formal notion of vector fields for tangentads, by isolating the correct universal property enjoyed by vector fields in ordinary tangent categories. We show that vector fields form a Lie algebra and a $2$-monad and show how to construct vector fields using PIE limits. Finally, we compute vector fields for some examples of tangentads. In a forthcoming paper, we extend the theory to other constructions: differential objects, differential bundles, and connections.
\end{abstract}

\begin{acknowledgements}
We are grateful to Steve Lack, who suggested an equivalent description of the universal properties of the formal constructions, which greatly improved the exposition and the clarity of this paper. We also thank Richard Garner, JS Lemay, and Michael Ching for useful discussions, corrections, and advice. Finally, we thank Quentin Schroeder and Luca Mesiti for useful discussions on PIE limits for fibrations. This material is based upon work supported by the AFOSR under award number FA9550-24-1-0008.
\end{acknowledgements}


\tableofcontents

\section{Introduction}
\label{section:introduction}
Differential geometry offers a long list of important geometric constructions such as vector fields, Euclidean spaces, vector bundles, principal bundles, connections, etc. A consistent component of tangent category theory aims to extend these constructions in the more general setting of tangent categories. Many efforts in this direction have been taken (see, for instance~\cite{cockett:tangent-cats},~\cite{cockett:differential-bundles},~\cite{cockett:connections}, or~\cite{cockett:differential-equations}), and more recently, some of these constructions have also been compared with related notions in algebraic geometry, adopting a tangent-categorical perspective (\cite{cruttwell:algebraic-geometry},~\cite{cruttwell:connections-algebraic-geometry}).
\par Interestingly, for each of these constructions, one assembles a new tangent category from a given one, so for instance, the vector fields of a tangent category form a new tangent category, and so do differential objects, connections, and so on.
\par Tangent category theory not only extends many constructions of differential geometry, but also comes in different categorical flavours. The first instance of this appeared in~\cite{cockett:differential-bundles} in which the notion of tangent fibrations, namely, fibrations between two tangent categories whose tangent bundle functors are compatible with the Cartesian lifts, was introduced. Similarly, the paper~\cite{cockett:tangent-monads} introduced the notion of tangent monads, which are monads over a tangent category with a distributive law between the monad and the tangent bundle functor, compatible with the tangent structure.
\par One is left to wonder how constructions like vector fields or differential bundles interact with the different flavours of tangent category theory. For instance, it is natural to ask whether or not a tangent monad on a tangent category $(\X,\TT)$ lifts to the tangent category of vector fields of $(\X,\TT)$, or if it does so on the tangent categories of differential bundles or of connections. Furthermore, one might like to ask the same question for other types of flavours, such as tangent fibrations or restriction tangent categories.
\par Our approach is rooted in the formal theory of tangentads, a notion firstly introduced in~\cite{lanfranchi:grothendieck-tangent-cats} and further studied in~\cite{lanfranchi:tangentads-I} as a formalization of tangent category theory where it was shown to capture a multitude of existing concepts, such as tangent monads, tangent fibrations, reverse tangent categories, strong display tangent categories, split restriction tangent categories, and infinitesimal objects.
\par To put it into a slogan, \textit{A tangentad is to a tangent category as a monad internal to a $2$-category is to a monad in ordinary category theory}.\\
This formal approach is a powerful tool to extend equivalences and constructions between $2$-categories into the tangent context. In particular, it was employed to prove the Grothendieck construction for tangent fibrations~\cite[Theorem~5.5]{lanfranchi:grothendieck-tangent-cats}, the construction of algebras for tangent monads~\cite[Theorem~4.4]{lanfranchi:tangentads-I}, and an equivalence between split restriction tangent categories and $\M$-tangent categories~\cite[Theorem~4.45]{lanfranchi:tangentads-I}.
\par The formal approach provides a unifying context to study this problem in a model-independent way. This saves one from defining similar constructions multiple times and from proving similar results in different contexts. For instance, it saves us from proving that a tangent monad sends the Lie bracket of two vector fields to another vector field. This is a direct consequence of the following two statements: (1) a tangent monad is an example of a tangentad in the $2$-category of monads; (2) the formal theory of vector fields for tangentads, developed in this paper, establishes a notion of Lie bracket for vector fields.
\par This paper, the first of a two-paper story, aims to develop a formal theory which generalizes the constructions of vector fields in the general context of tangentads. To begin, we isolate the correct universal property satisfied by this construction in the context of tangent categories and use this characterization to define vector fields in the general context of tangentads.
\par In particular, we show that the tangent categories of vector fields of a given tangent category come equipped with extra structure which makes a suitable $2$-functor from the $2$-category of tangent categories to itself, corepresentable. The $2$-functor which classifies vector fields on a tangent category $(\X,\TT)$ is the $2$-functor which sends a tangent category $(\X',\TT')$ to the tangent category $\VF[\X',\TT'|\X,\TT]$ of vector fields of the Hom-tangent category of tangent morphisms from $(\X',\TT')$ to $(\X,\TT)$. Therefore, we say, a tangentad in a $2$-category admits vector fields when the corresponding $2$-functor from the $2$-category of tangentads to the $2$-category of tangent categories is corepresentable in a suitable sense.
\par We extend many of the constructions and results of vector fields for ordinary tangent category theory to general tangentads. Here is a list of some of these results:
\begin{itemize}
\item Proposition~\ref{proposition:VF-functoriality} proves that the construction of vector fields is $2$-functorial;

\item Theorem~\ref{theorem:monad-structure-vector-fields} shows that the tangentad $\VF(\X,\TT)$ of vector fields of a tangentad $(\X,\TT)$ carries the structure of a commutative monoid;

\item Theorem~\ref{theorem:Lie-algebra-vector-fields} shows that the tangentad $\VF(\X,\TT)$ of vector fields of a tangentad $(\X,\TT)$ with negatives carries the structure of a Lie algebra;

\item Theorem~\ref{theorem:monad-structure-vector-fields} shows that the $2$-functor $\VF\colon\Tng(\CC)\to\Tng(\CC)$ is a $2$-monad.
\end{itemize}
After discussing in Section~\ref{section:PIE-limits} how to lift PIE limits to the $2$-category of tangentads internal to a given $2$-category, we employ PIE limits to construct vector fields.
Finally, we study the formal construction of vector fields for tangent monads, tangent fibrations, tangent indexed categories, tangent split restriction categories and discuss how to extend it to general tangent (non-necessarily split) restriction categories via suitable pullbacks.
\par In the follow-up paper, we will study the constructions of differential objects, differential bundles, and connections and extend them to the formal theory of tangentads.


\section{Recollection on tangentads}
\label{section:tangentads}
Tangentads were initially introduced in~\cite{lanfranchi:grothendieck-tangent-cats} with the name of tangent objects, to establish a Grothendieck construction in the context of tangent categories. In~\cite{lanfranchi:tangentads-I}, this notion was revisited and proved to be suitable to capture a long list of existing concepts in tangent category theory, such as tangent fibrations, tangent monads, tangent split restriction categories, reverse tangent categories, strong display tangent categories, and infinitesimal objects. In this section, we briefly review the definition of a tangentad in a $2$-category.

\subsection{A brief introduction to tangent categories}
\label{subsection:tangent-categories}
Tangent categories were introduced by Rosick\'y in~\cite{rosicky:tangent-cats}, and later revisited and generalized by Cockett and Cruttwell in~\cite{cockett:tangent-cats}, as a categorical context for differential geometry. This section is dedicated to recalling this definition. For starters, recall the definition of an additive bundle and related morphisms.

\begin{definition}
\label{definition:additive-bundle}
In a category $\X$, an \textbf{additive bundle} consists of a pair of objects $B$ and $E$ together with three morphisms $q\colon E\to B$, $z_q\colon B\to E$, and $s_q\colon E_2\to E$, respectively called the \textbf{projection}, the \textbf{zero morphism}, and the \textbf{sum morphism}, where $E_n$ denotes the $n$-fold pullback of $q$ along itself. Moreover, $q$, $z_q$, and $s_q$ need to satisfy the following conditions:
\begin{enumerate}
\item For every positive integer $n$, the $n$-fold pullback
\begin{equation*}
\begin{tikzcd}
{E_n} & E \\
E & M
\arrow["{\pi_n}", from=1-1, to=1-2]
\arrow["{\pi_1}"', from=1-1, to=2-1]
\arrow["\lrcorner"{anchor=center, pos=0.125}, draw=none, from=1-1, to=2-2]
\arrow["\dots"{marking, allow upside down}, shift right=3, draw=none, from=1-2, to=2-1]
\arrow["q", from=1-2, to=2-2]
\arrow["q"', from=2-1, to=2-2]
\end{tikzcd}
\end{equation*}
of the projection $q$ along itself exists;

\item The zero morphism is a section of the projection:
\begin{equation*}
\begin{tikzcd}
M & E \\
& M
\arrow["{z_q}", from=1-1, to=1-2]
\arrow[equals, from=1-1, to=2-2]
\arrow["q", from=1-2, to=2-2]
\end{tikzcd}
\end{equation*}

\item The sum morphism is compatible with the projection. Moreover, $s_q$ is associative, unital, and commutative
\begin{equation*}
\begin{tikzcd}
{E_2} & E \\
E & M
\arrow["{s_q}", from=1-1, to=1-2]
\arrow["{\pi_1}"', from=1-1, to=2-1]
\arrow["q", from=1-2, to=2-2]
\arrow["q"', from=2-1, to=2-2]
\end{tikzcd}\hfill\quad
\begin{tikzcd}
{E_3} & {E_2} \\
{E_2} & E
\arrow["{\id_E\times_Ms_q}", from=1-1, to=1-2]
\arrow["{s_q\times_M\id_E}"', from=1-1, to=2-1]
\arrow["{s_q}", from=1-2, to=2-2]
\arrow["{s_q}"', from=2-1, to=2-2]
\end{tikzcd}\hfill\quad
\begin{tikzcd}
E & {E_2} \\
& E
\arrow["{\<z_q\o q,\id_E\>}", from=1-1, to=1-2]
\arrow[equals, from=1-1, to=2-2]
\arrow["{s_q}", from=1-2, to=2-2]
\end{tikzcd}\hfill\quad
\begin{tikzcd}
{E_2} & {E_2} \\
& E
\arrow["\tau", from=1-1, to=1-2]
\arrow["{s_q}"', from=1-1, to=2-2]
\arrow["{s_q}", from=1-2, to=2-2]
\end{tikzcd}
\end{equation*}
where $\tau\colon E_2\to E_2$ denotes the canonical symmetry $\pi_2\times\pi_1$.
\end{enumerate}
In the following, an additive bundle $(q\colon E\to B,z_q,s_q)$ will be denoted by $\q\colon E\to B$ or by $\q^E_B$. When the symbol that denotes the additive bundle is decorated with a superscript or a subscript, the same decoration is applied to the projection, the zero morphism, and the sum morphism. For instance, the additive bundle $\q'\colon E'\to B'$ consists of the projection $q'$, the zero morphism $z_q'$, and the sum morphism $s_q'$.
\end{definition}

\begin{definition}
\label{definition:additive-bundle-morphism}
An \textbf{additive bundle morphism} from an additive bundle $\q\colon E\to B$ to an additive bundle $\q'\colon E'\to B'$ consists of a pair of morphisms $f\colon B\to B'$ and $g\colon E\to E'$ satisfying the following conditions:
\begin{equation*}
\begin{tikzcd}
E & {E'} \\
B & {B'}
\arrow["g", from=1-1, to=1-2]
\arrow["q"', from=1-1, to=2-1]
\arrow["{q'}", from=1-2, to=2-2]
\arrow["f"', from=2-1, to=2-2]
\end{tikzcd}\hfill\quad
\begin{tikzcd}
E & {E'} \\
B & {B'}
\arrow["g", from=1-1, to=1-2]
\arrow["{z_q}", from=2-1, to=1-1]
\arrow["f"', from=2-1, to=2-2]
\arrow["{z_q'}"', from=2-2, to=1-2]
\end{tikzcd}\hfill\quad
\begin{tikzcd}
{E_2} & {E'_2} \\
E & {E'}
\arrow["{g\times_fg}", from=1-1, to=1-2]
\arrow["{s_q}"', from=1-1, to=2-1]
\arrow["{s_q'}", from=1-2, to=2-2]
\arrow["g"', from=2-1, to=2-2]
\end{tikzcd}
\end{equation*}
An additive bundle morphism is denoted by $(f,g)\colon\q\to\q'$, where, in the first position, $f$ indicates the base morphism and in the second one, $g$ indicates the top morphism. 
\end{definition}

The ``correct'' notion of limits in a tangent category is the notion of $\T$-limits.

\begin{definition}
\label{definition:tangent-limit}
Let $\X$ be a category and $\T\colon\X\to\X$ be an endofunctor of $\X$. A \textbf{$\T$-limit diagram} consists of a limit diagram in $\X$ whose universal property is preserved by all iterates $\T^n$ of $\T$. In the following, we adopt the convention to call \textbf{$\T$-pullbacks}, \textbf{$\T$-equalizers}, etc., the $\T$-limit diagrams which, as limit diagrams, are respectively pullbacks, equalizers, etc. When the endofunctor $\T$ represents the tangent bundle functor of a tangent structure, $\T$-limits are also called \textbf{tangent limits} and the convention extends to $\T$-pullbacks, $\T$-equalizers, etc.
\end{definition}

We can now recall the definition of tangent categories as formulated in~\cite[Definition~2.3]{cockett:tangent-cats}.

\begin{definition}
\label{definition:tangent-category}
A \textbf{tangent structure} on a category $\X$ consists of the following data:
\begin{description}
\item[Tangent bundle functor] An endofunctor $\T\colon\X\to\X$;

\item[Projection]  A natural transformation $p_M\colon\T M\to M$, natural in $M$, which admits all $n$-fold $\T$-pullbacks
\begin{equation*}
\begin{tikzcd}
{\T_nM} & {\T M} \\
{\T M} & M
\arrow["{\pi_n}", from=1-1, to=1-2]
\arrow["{\pi_1}"', from=1-1, to=2-1]
\arrow["\lrcorner"{anchor=center, pos=0.125}, draw=none, from=1-1, to=2-2]
\arrow["\dots"{marking, allow upside down}, shift right=3, draw=none, from=1-2, to=2-1]
\arrow["{p_M}", from=1-2, to=2-2]
\arrow["{p_M}"', from=2-1, to=2-2]
\end{tikzcd}
\end{equation*}
along itself;

\item[Zero morphism] A natural transformation $z_M\colon M\to\T M$, natural in $M$;

\item[Sum morphism] A natural transformation $s_M\colon\T_2M\to\T M$, natural in $M$;
\end{description}
such that, for each $M\in\X$, $\p_M\=(p_M,z_M,s_M)$ is an additive bundle of $\X$;
\begin{description}
\item[Vertical lift] A natural transformation $l_M\colon\T M\to\T^2M$, natural in $M$, where $\T^2M\=\T\T M$, such that:
\begin{align*}
&(z_M,l_M)\colon\p_M\to\T\p_M
\end{align*}
is an additive bundle morphism;

\item[Canonical flip] A natural transformation $c_M\colon\T^2M\to\T^2M$ such that:
\begin{align*}
&(\id_M,c_M)\colon\T\p_M\to\p_{\T M}
\end{align*}
is an additive bundle morphism;
\end{description}
Moreover, the following coherence conditions must hold:
\begin{enumerate}
\item The vertical lift is coassocative, and compatible with the canonical flip:
\begin{equation*}
\begin{tikzcd}
{\T M} & {\T^2M} \\
{\T^2M} & {\T^3M}
\arrow["l_M", from=1-1, to=1-2]
\arrow["lM"', from=1-1, to=2-1]
\arrow["{\T l_M}", from=1-2, to=2-2]
\arrow["{l_{\T M}}"', from=2-1, to=2-2]
\end{tikzcd}\hfill\quad
\begin{tikzcd}
{\T M} & {\T^2M} \\
& {\T^2M}
\arrow["l_M", from=1-1, to=1-2]
\arrow["l_M"', from=1-1, to=2-2]
\arrow["c_M", from=1-2, to=2-2]
\end{tikzcd}\hfill\quad
\begin{tikzcd}
{\T^2M} & {\T^3M} & {\T^3M} \\
{\T^2M} && {\T^3M}
\arrow["{l_{\T M}}", from=1-1, to=1-2]
\arrow["c_M"', from=1-1, to=2-1]
\arrow["{\T c_M}", from=1-2, to=1-3]
\arrow["{c_{\T M}}", from=1-3, to=2-3]
\arrow["{\T l_M}"', from=2-1, to=2-3]
\end{tikzcd}
\end{equation*}

\item The canonical flip is a symmetric braiding:
\begin{equation*}
\begin{tikzcd}
{\T^2M} & {\T^2M} \\
& {\T^2M}
\arrow["c_M", from=1-1, to=1-2]
\arrow[equals, from=1-1, to=2-2]
\arrow["c_M", from=1-2, to=2-2]
\end{tikzcd}\hfill\quad
\begin{tikzcd}
{\T^3M} & {\T^3M} & {\T^3M} \\
{\T^3M} & {\T^3M} & {\T^3M}
\arrow["{\T c_M}", from=1-1, to=1-2]
\arrow["{c_{\T M}}"', from=1-1, to=2-1]
\arrow["{c_{\T M}}", from=1-2, to=1-3]
\arrow["{\T c_M}", from=1-3, to=2-3]
\arrow["{\T c_M}"', from=2-1, to=2-2]
\arrow["{c_{\T M}}"', from=2-2, to=2-3]
\end{tikzcd}
\end{equation*}

\item The tangent bundle is locally linear, namely, the following is a pullback diagram:
\begin{equation*}
\begin{tikzcd}[column sep=huge]
{\T_2M} & {\T\T_2M} & {\T^2M} \\
M && {\T M}
\arrow["{z_{\T M}\times l_M}", from=1-1, to=1-2]
\arrow["{\pi_1p_M}"', from=1-1, to=2-1]
\arrow["{\T s_M}", from=1-2, to=1-3]
\arrow["{\T p_M}", from=1-3, to=2-3]
\arrow["z_M"', from=2-1, to=2-3]
\end{tikzcd}
\end{equation*}
\end{enumerate}
A \textbf{tangent category} consists of a category $\X$ equipped with a tangent structure $\TT\=(\T,p,z,s,l,c)$. Furthermore, a tangent structure \textbf{has negatives} when it is equipped with:
\begin{description}
\item[negation] A natural transformation $n_M\colon\T M\to\T M$, natural in $M$, satisfying the following condition:
\begin{equation*}
\begin{tikzcd}[column sep=huge]
\T M & {\T_2M} \\
{\id_M} & \T M
\arrow["{\<n_M,\id_{\T M}\>}", from=1-1, to=1-2]
\arrow["p_M"', from=1-1, to=2-1]
\arrow["s_M", from=1-2, to=2-2]
\arrow["z_M"', from=2-1, to=2-2]
\end{tikzcd}
\end{equation*}
\end{description}
\end{definition}

In the following, we denote the tangent bundle functor of a tangent structure with the same symbol adopted to denote a tangent structure, e.g., $\TT$, in the font $\T$. The projection, the zero morphism, the sum morphism, the vertical lift, and the canonical flip are respectively denoted by the letters $p$, $z$, $s$, $l$, and $c$, and when the tangent structure has negatives, the negation is denoted by $n$. When the symbol denoting the tangent structure is decorated with a subscript or with a superscript, the tangent bundle functor and each of the structural natural transformations are decorated in the same way, e.g., for a tangent structure $\TT'_\o$, the tangent bundle functor is denoted by $\T'_\o$, the projection by $p'_\o$ and so on.
\par Examples of tangent categories can be found in~\cite{cockett:tangent-cats},~\cite{cockett:differential-bundles}, and~\cite{ikonicoff:operadic-algebras-tagent-cats}. Here is a brief list of selected examples from these resources:
\begin{itemize}
\item Any category with tangent bundle functor and structural natural transformations as identities;

\item The category of (finitely-dimensional) smooth manifolds;

\item The subcategory of microlinear objects in a model of synthetic differential geometry;

\item Any Cartesian differential Category;

\item The category of (affine) schemes over a commutative and unital ring;

\item The category of algebras of a symmetric operad over the symmetric monoidal category of $R$-modules and its opposite category.
\end{itemize}

\subsection{The definition of a tangentad}
\label{subsection:definition-tangentad}
A tangentad consists of an object in a $2$-category together with a structure which formalizes a tangent structure on a category. In particular, tangentads in the $2$-category $\Cat$ of categories are precisely tangent categories. One way to define such a formalization is via Leung's approach to tangent categories. In~\cite{leung:weil-algebras}, Leung showed that tangent structures $\TT$ on a category $\X$ are in correspondence with strong monoidal functors $\Leung[\TT]$ from a monoidal category $\Weil$ to the monoidal category $\End(\X)$ of endofunctors of $\X$, satisfying an extra technical condition. We leave the reader to consult~\cite{leung:weil-algebras} for a complete explanation of this result.
\par For the sake of completeness, here we briefly give a sketch of this correspondence. Firstly, define $\Weil$ as a monoidal category generated by the following family of rigs, a.k.a. semirings
\begin{align*}
&W_n\=\N[x_1\,x_n]/(x_ix_j;i,j=1\,n)
\end{align*}
where the monoidal product is the tensor product of rigs, namely, the pushout in the category of rigs. Moreover, the morphisms of $\Weil$ are generated by the following rig morphisms
\begin{align*}
&p\colon W\to\N             &&z\colon\N\to W\\
&s\colon W_2\to W           &&l\colon W\to W\x W\\
&c\colon W\x W\to W\x W
\end{align*}
where $W=W_1$, $p$ is the augmentation map, namely, $p(a+b\epsilon)\=a$, for each $a+b\epsilon\in W$, where $\epsilon^2=0$; $z$ is the inclusion, namely, $z(a)\=a$; $s$ is the sum, namely, $s(\epsilon_1)=\epsilon=s(\epsilon_2)$; $l$ sends $\epsilon$ to $\epsilon\x\epsilon$; finally, $c$ is the canonical symmetry of the tensor product.
\par A tangent structure $\TT$ on a category $\X$ corresponds to a strong monoidal functor
\begin{align*}
&\Leung[\TT]\colon\Weil\to\End_\CC(\X)
\end{align*}
which sends $W$ to the tangent bundle functor $\T$, each $W_n$ to the $n$-fold pullback of the projection along itself, namely, $\Leung[\TT](W_n)\=\T_n$; $p$ to the projection $p\colon\T\to\id_\X$; $z$ to the zero morphism $z\colon\id_\X\to\T$; $s$ to the sum morphism $s\colon\T_2\to\T$; $l$ to the vertical lift $l\colon\T\to\T^2$; and $c$ to the canonical flip $c\colon\T^2\to\T^2$.
\par Adopting Leung's perspective, the definition of a tangentad becomes straightforward.

\begin{definition}
\label{definition:tangentad}
A \textbf{tangentad} in a $2$-category $\CC$ consists of an object $\X$ of $\CC$ together with a strong monoidal functor
\begin{align*}
&\Leung[\TT]\colon\Weil\to\End_\CC(\X)
\end{align*}
from the monoidal category $\Weil$ to the monoidal category of $1$-endomorphisms of $\X$ in $\CC$, which preserves the fundamental tangent limits in a pointwise way, namely, for every morphism $F\colon\X'\to\X$, the functor
\begin{align*}
&\Weil\xrightarrow{\Leung[\TT]}\End(\X)=\CC(\X,\X)\xrightarrow{\CC(F,\X)}\CC(\X',\X)
\end{align*}
preserves the following pullback diagrams:
\begin{equation*}
\begin{tikzcd}
{W^{\x n}\x W_n} & {W^{\x n}\x W} \\
{W^{\x n}\x W} & {W^{\x n}}
\arrow["{W^{\x n}\x\pi_n}", from=1-1, to=1-2]
\arrow["{W^{\x n}\x\pi_1}"', from=1-1, to=2-1]
\arrow["\lrcorner"{anchor=center, pos=0.125}, draw=none, from=1-1, to=2-2]
\arrow["{W^{\x n}\x p}", from=1-2, to=2-2]
\arrow["\dots"{marking, allow upside down}, shift left=4, draw=none, from=2-1, to=1-2]
\arrow["{W^{\x n}\x p}"', from=2-1, to=2-2]
\end{tikzcd}\hfill\quad
\begin{tikzcd}
{W_2} && {W\x W_2} & {W\x W} \\
\N &&& W
\arrow["{\<l\o\pi_1,(z\x W)\o\pi_2\>}", from=1-1, to=1-3]
\arrow["{\pi_1p}"', from=1-1, to=2-1]
\arrow["{W\x s}", from=1-3, to=1-4]
\arrow["{W\x p}", from=1-4, to=2-4]
\arrow[""{name=0, anchor=center, inner sep=0}, "z"', from=2-1, to=2-4]
\arrow["\lrcorner"{anchor=center, pos=0.125}, draw=none, from=1-1, to=0]
\end{tikzcd}
\end{equation*}
\end{definition}

\begin{remark}
\label{remark:poon-and-christine-approaches}
In~\cite[Definition~14.2]{leung:weil-algebras}, Leung already suggested a generalization of a tangent category, by introducing the notion of a \textit{tangent structure internal} to a monoidal category $\mathcal{E}$ as a strong monoidal functor $\Leung[\TT]\colon\Weil\to\mathcal{E}$ which preserves the fundamental tangent limits in a pointwise way. In this sense, a tangentad consists of an object $\X$ in a $2$-category $\CC$ together with a tangent structure internal to the monoidal category $\End(\X)$.
\par Another approach to generalize tangent structures was employed in~\cite{bauer:infinity-tangent-cats} in which $\infty$-tangent category is defined as a strong monoidal functor from the $\infty$-category of Weil algebras $\Weil^\infty$ to the monoidal category $\End(\X)$ of endomorphisms of an $\infty$-category $\X$.
\par In an informal discussion, Michael Ching points out that $\infty$-tangent categories can be generalized to strong monoidal functors from $\Weil^\infty$ to the monoidal category $\End_\CC(\X)$, where $\CC$ is a $(\infty,2)$-category, which preserve suitable limits. Such functors split along the homotopy category of $\Weil^\infty$, which coincides with $\Weil$. Thus, every such functor gives rise to a tangentad.
\end{remark}

Before proceeding further, let us unpack Definition~\ref{definition:tangentad}. A tangentad in a $2$-category $\CC$ consists of:
\begin{description}
\item[Base object] An object $\X$;

\item[Tangent bundle $1$-morphism] A $1$-endomorphism $\T\colon\X\to\X$;

\item[Projection]  A $2$-morphism $p\colon\T\Rightarrow\id_\X$, which admits all pointwise $n$-fold pullbacks
\begin{equation*}
\begin{tikzcd}
	{\T_n} & \T \\
	\T & {\id_\X}
	\arrow["{\pi_n}", from=1-1, to=1-2]
	\arrow["{\pi_1}"', from=1-1, to=2-1]
	\arrow["\lrcorner"{anchor=center, pos=0.125}, draw=none, from=1-1, to=2-2]
	\arrow["\dots"{marking, allow upside down}, shift right=3, draw=none, from=1-2, to=2-1]
	\arrow["p", from=1-2, to=2-2]
	\arrow["p"', from=2-1, to=2-2]
\end{tikzcd}
\end{equation*}
which are preserved  by all iterates $\T^n$ of $\T$;

\item[Zero morphism] A $2$-morphism $z\colon\id_\X\Rightarrow\T$;

\item[Sum morphism] A $2$-morphism $s\colon\T_2\Rightarrow\T$;
\end{description}
such that, $\p\=(p,z,s)$ is a $\bar\T$-additive bundle of $\End(\X)$, where $\bar\T$ sends an endomorphism $F\colon\X\to\X$ to $\T\o F$;
\begin{description}
\item[Vertical lift] A $2$-morphism $l\colon\T\Rightarrow\T^2$ such that:
\begin{align*}
&(z,l)\colon\p\to\T\p
\end{align*}
is an additive bundle morphism;

\item[Canonical flip] A $2$-morphism $c\colon\T^2\Rightarrow\T^2$ such that:
\begin{align*}
&(\id_\X,c)\colon\T\p\to\p_\T
\end{align*}
is an additive bundle morphism;
\end{description}
Moreover, the following coherence conditions must hold:
\begin{enumerate}
\item The vertical lift is coassocative, and compatible with the canonical flip:
\begin{equation*}
\begin{tikzcd}
{\T} & {\T^2} \\
{\T^2} & {\T^3}
\arrow["l", from=1-1, to=1-2]
\arrow["l"', from=1-1, to=2-1]
\arrow["{\T l}", from=1-2, to=2-2]
\arrow["{l_\T}"', from=2-1, to=2-2]
\end{tikzcd}\hfill\quad
\begin{tikzcd}
{\T} & {\T^2} \\
& {\T^2}
\arrow["l", from=1-1, to=1-2]
\arrow["l"', from=1-1, to=2-2]
\arrow["c", from=1-2, to=2-2]
\end{tikzcd}\hfill\quad
\begin{tikzcd}
{\T^2} & {\T^3} & {\T^3} \\
{\T^2} && {\T^3}
\arrow["{l_\T}", from=1-1, to=1-2]
\arrow["c"', from=1-1, to=2-1]
\arrow["{\T c}", from=1-2, to=1-3]
\arrow["{c_\T}", from=1-3, to=2-3]
\arrow["{\T l}"', from=2-1, to=2-3]
\end{tikzcd}
\end{equation*}

\item The canonical flip is a symmetric braiding:
\begin{equation*}
\begin{tikzcd}
{\T^2} & {\T^2} \\
& {\T^2}
\arrow["c", from=1-1, to=1-2]
\arrow[equals, from=1-1, to=2-2]
\arrow["c", from=1-2, to=2-2]
\end{tikzcd}\hfill\quad
\begin{tikzcd}
{\T^3} & {\T^3} & {\T^3} \\
{\T^3} & {\T^3} & {\T^3}
\arrow["{\T c}", from=1-1, to=1-2]
\arrow["{c_\T}"', from=1-1, to=2-1]
\arrow["{c_\T}", from=1-2, to=1-3]
\arrow["{\T c}", from=1-3, to=2-3]
\arrow["{\T c}"', from=2-1, to=2-2]
\arrow["{c_\T}"', from=2-2, to=2-3]
\end{tikzcd}
\end{equation*}

\item The tangent bundle is locally linear, namely, the following is a pointwise pullback diagram:
\begin{equation*}
\begin{tikzcd}
{\T_2} & {\T\o\T_2} & {\T^2} \\
\id_\X && {\T}
\arrow["{z_\T\times l}", from=1-1, to=1-2]
\arrow["{\pi_1p}"', from=1-1, to=2-1]
\arrow["{\T s}", from=1-2, to=1-3]
\arrow["{\T p}", from=1-3, to=2-3]
\arrow["z"', from=2-1, to=2-3]
\end{tikzcd}
\end{equation*}
\end{enumerate}
Furthermore, a tangentad admits negatives when it is equipped with:
\begin{description}
\item[Negation] A $2$-morphism $n\colon\T\Rightarrow\T$ such that:
\begin{equation*}
\begin{tikzcd}
\T & {\T_2} \\
{\id_\X} & \T
\arrow["{\<n,\id_\T\>}", from=1-1, to=1-2]
\arrow["p"', from=1-1, to=2-1]
\arrow["s", from=1-2, to=2-2]
\arrow["z"', from=2-1, to=2-2]
\end{tikzcd}
\end{equation*}
\end{description}

\subsection{The \texorpdfstring{$2$-categories}{2-categories} of tangentads}
\label{subsection:2-categories-tangentads}
The tangentads of a $2$-category $\CC$ can be organized in four different $2$-categories. In this section, we recall the notion of $1$ and $2$-morphisms between tangentads. Let us start with $1$-morphisms.

\begin{definition}
\label{definition:tangent-morphisms}
Let $(\X,\TT)$ and $(\X',\TT')$ be two tangentads of $\CC$. A \textbf{lax tangent $1$-morphism} from $(\X,\TT)$ to $(\X',\TT')$ consists of a $1$-morphism $F\colon\X\to\X'$ together with a $2$-morphism $\alpha\colon F\o\T\to\T'\o F$, called a \textbf{lax distributive law} satisfying the following conditions:
\begin{enumerate}
\item Additivity:
\begin{equation*}
\begin{tikzcd}
{F\o\T} & {\T'\o F} \\
F & F
\arrow["\alpha", from=1-1, to=1-2]
\arrow["Fp"', from=1-1, to=2-1]
\arrow["{p'F}", from=1-2, to=2-2]
\arrow[equals, from=2-1, to=2-2]
\end{tikzcd}\hfill\quad
\begin{tikzcd}
{F\o\T} & {\T'\o F} \\
F & F
\arrow["\alpha", from=1-1, to=1-2]
\arrow["Fz", from=2-1, to=1-1]
\arrow[equals, from=2-1, to=2-2]
\arrow["{z'F}"', from=2-2, to=1-2]
\end{tikzcd}\hfill\quad
\begin{tikzcd}
{F\o\T_2} && {\T'_2\o F} \\
{F\o\T} && {\T'\o F}
\arrow["{\<\alpha\o F\pi_1,\alpha\o F\pi_2\>}", from=1-1, to=1-3]
\arrow["Fs"', from=1-1, to=2-1]
\arrow["{s'F}", from=1-3, to=2-3]
\arrow["\alpha"', from=2-1, to=2-3]
\end{tikzcd}
\end{equation*}

\item Compatibility with the vertical lift:
\begin{equation*}
\begin{tikzcd}
{F\o\T^2} & {\T'\o F\o\T} & {\T'^2\o F} \\
{F\o\T} && {\T'\o F}
\arrow["{\alpha\T}", from=1-1, to=1-2]
\arrow["{\T'\alpha}", from=1-2, to=1-3]
\arrow["Fl", from=2-1, to=1-1]
\arrow["\alpha"', from=2-1, to=2-3]
\arrow["{l'F}"', from=2-3, to=1-3]
\end{tikzcd}
\end{equation*}

\item Compatibility with the canonical flip:
\begin{equation*}
\begin{tikzcd}
{F\o\T^2} & {\T'\o F\o\T} & {\T'^2\o F} \\
{F\o\T^2} & {\T'\o F\o\T} & {\T'^2\o F}
\arrow["{\alpha\T}", from=1-1, to=1-2]
\arrow["{\T'\alpha}", from=1-2, to=1-3]
\arrow["Fc", from=2-1, to=1-1]
\arrow["{\alpha\T}"', from=2-1, to=2-2]
\arrow["{\T'\alpha}"', from=2-2, to=2-3]
\arrow["{c'F}"', from=2-3, to=1-3]
\end{tikzcd}
\end{equation*}
\end{enumerate}
A \textbf{colax tangent $1$-morphism} from $(\X,\TT)$ to $(\X',\TT')$ consists of a $1$-morphism $G\colon\X\to\X'$ together with a $2$-morphism $\beta\colon\T'\o G\to G\o\T$, called a \textbf{colax distributive law}, which satisfies the dual conditions of a lax distributive law, namely, additivity and the compatibility with the vertical lift and the canonical flip, in which the direction of the distributive law is now reversed.
\par A \textbf{strong tangent $1$-morphism} is a lax tangent $1$-morphism whose distributive law is invertible. A \textbf{strict tangent $1$-morphism} is a strong tangent $1$-morphism whose distributive law is the identity.\newline
Finally, a tangent morphism $(F,\alpha)\colon(\X,\TT)\to(\X',\TT')$ \textbf{preserves the fundamental tangent limits} when the functor $\CC(\X,F)\colon\End(\X)\to\CC(\X,\X')$ preserves the fundamental tangent limits of $(\X,\TT)$ in a pointwise way.
A lax tangent $1$-morphism is denoted by $(F,\alpha)\colon(\X,\TT)\to(\X',\TT')$ while for colax tangent $1$-morphisms we adopt the notation $(F,\alpha)\colon(\X,\TT)\nto(\X',\TT')$. Finally, a strict tangent $1$-morphism is denoted by its underlying $1$-morphism.
\end{definition}

We can now recall the notion of $2$-morphisms of tangentads.

\begin{definition}
\label{definition:tangent-2-morphisms}
Given two lax tangent morphisms $(F,\alpha),(G,\beta)\colon(\X,\TT)\to(\X',\TT')$ of tangentads in $\CC$, a \textbf{tangent $2$-morphism} from $(F,\alpha)$ to $(G,\beta)$ consists of a $2$-morphism $\varphi\colon F\Rightarrow G$ satisfying the following compatibility between the distributive laws:
\begin{equation*}
\begin{tikzcd}
{F\o\T} & {\T'\o F} \\
{G\o\T} & {\T'\o G}
\arrow["\alpha", from=1-1, to=1-2]
\arrow["{\varphi\T}"', from=1-1, to=2-1]
\arrow["{\T'\varphi}", from=1-2, to=2-2]
\arrow["\beta"', from=2-1, to=2-2]
\end{tikzcd}
\end{equation*}
A tangent $2$-morphism between two colax tangent morphisms $(F,\alpha),(G,\beta)\colon(\X,\TT)\nto(\X',\TT')$ consists of a $2$-morphism $\varphi\colon F\Rightarrow G$ which satisfies the same compatibility as a tangent $2$-morphism between lax tangent morphisms, where the distributive laws are reversed.
\end{definition}

Tangentads with their tangent $1$ and $2$-morphisms can be organized into four distinct $2$-categories:
\begin{itemize}
\item Tangentads of $\CC$, lax tangent morphisms, and corresponding $2$-morphisms form the $2$-category $\Tng(\CC)$;

\item Tangentads of $\CC$, colax tangent morphisms, and corresponding $2$-morphisms form the $2$-category $\Tng_\co(\CC)$;

\item Tangentads of $\CC$, strong tangent morphisms, and corresponding $2$-morphisms form the $2$-category $\Tng_\cong(\CC)$;

\item Tangent categories, strict tangent morphisms, and corresponding $2$-morphisms form the $2$-category $\Tng_=(\CC)$.
\end{itemize}

\subsection{Examples of tangentads}
\label{subsection:examples-tangentads}
In~\cite{lanfranchi:tangentads-I}, we explored a long list of tangentads. We now briefly recall some of these examples.

\begin{example}
\label{example:tangent-categories}
Tangent categories are the archetypal example of tangentads. In particular, they are tangentads in the $2$-category $\Cat$ of categories.
\end{example}

\begin{example}
\label{example:tangent-monads}
Tangent monads, introduced in~\cite[Definition~19]{cockett:tangent-monads}, are monads in the $2$-category of tangent categories. Concretely, a tangent monad on a tangent category $(\X,\TT)$ consists of a monad $S\colon\X\to\X$ equipped with a natural transformation $\alpha_M\colon S\T M\to\T SM$, for $M\in\X$, such that $(S,\alpha)\colon(\X,\TT)\to(\X,\TT)$ is a lax tangent morphism. Furthermore, the unit $\eta_M\colon M\to SM$ and the multiplication $\mu_M\colon S^2M\to SM$ of the monad $S$ are tangent natural transformations, namely, they are compatible with $\alpha$.
\par More generally, a tangent monad internal to a $2$-category $\CC$ is monad in the $2$-category $\Tng(\CC)$ of tangentads of $\CC$.
\par \cite[Proposition~4.2]{lanfranchi:tangentads-I} shows that, for a $2$-category $\CC$, the $2$-category $\Mnd(\Tng(\CC))$ of monads in the $2$-category of tangentads of $\CC$ is isomorphic to the $2$-category $\Tng(\Mnd(\CC))$ of tangentads in the $2$-category of monads of $\CC$. In particular, tangent monads are tangentads in the $2$-category $\Mnd\=\Mnd(\Cat)$ of monads.
\end{example}

\begin{example}
\label{example:tangent-fibrations}
Tangent fibrations, introduced in~\cite[Definition~5.2]{cockett:differential-bundles}, are (cloven) fibrations\linebreak $\Pi\colon(\X',\TT')\to(\X,\TT)$ between two tangent categories which preserve the tangent structures strictly, namely, $\Pi$ is a strict tangent morphism, and whose tangent bundle functors define an endomorphism $(\T,\T')\colon\Pi\to\Pi$ of fibrations, namely, the tangent bundle functor $\T'$ sends each Cartesian lift $\varphi_f\colon f^\*E\to E$ of each morphism $f\colon A\to B$ of $(\X,\TT)$ and each $E\in\Pi^{-1}(B)$ in the fibre over $B$, to a Cartesian lift $\T'(\varphi_f)\colon\T'(f^\*E)\to\T'E$ of $\T f$ onto $\T'E$.
\par \cite[Proposition~5.1]{lanfranchi:grothendieck-tangent-cats} proves that tangent fibrations are tangentads in the $2$-category $\Fib$. Concretely, the objects of $\Fib$ are triples $(\X,\X';\Pi)$ formed by two categories $\X$ and $\X'$ and a fibration $\Pi\colon\X'\to\X$. $1$-morphisms of $\Fib$ $(F,F')\colon(\X_\o,\X_\o';\Pi_\o)\to(\X_\b,\X'_\b;\Pi_\b)$ are pairs of functors $F\colon\X_\o\to\X_\b$ and $F'\colon\X_\o'\to\X_\b'$ which strictly commute with the fibrations, namely, $\Pi_\b\o F'=F\o\Pi_\o$, and preserve the Cartesian lifts, namely, $F'$ sends each Cartesian lift $\varphi_f\colon f^\*E\to E$ of a morphism $f\colon A\to B$ of $\X_\o$ to a Cartesian lift of $Ff$. Finally, $2$-morphisms of $\Fib$
\begin{align*}
&(\varphi,\varphi')\colon(F,F')\Rightarrow(G,G')\colon(\X_\o,\X_\o';\Pi_\o)\to(\X_\b,\X_\b';\Pi_\b)
\end{align*}
are pairs of natural transformations $\varphi\colon F\Rightarrow G$ and $\varphi'\colon F'\Rightarrow G'$ which commute with the fibrations, namely, $\Pi_\b\varphi'=\varphi_{\Pi_\o}$.
\end{example}

\begin{example}
\label{example:tangent-indexed-categories}
Tangent indexed categories were introduced in~\cite[Definition~5.3]{lanfranchi:grothendieck-tangent-cats} as tangentads in the $2$-category $\Indx$ of indexed categories. Concretely, the objects of $\Indx$ are pairs $(\X,\I)$ formed by a category $\X$ and an indexed category, a.k.a., a pseudofunctor $\X^\op\to\Cat$. The $1$-morphisms $(F,F',\xi)\colon(\X_\o,\I_\o)\to(\X_\b,\I_\b)$ of $\Indx$ consist of a functor $F\colon\X_\o\to\X_\b$, a collection $F'$ of functors $F_A'\colon\I_\o(A)\to\I_\b(FA)$, indexed by the objects $A$ of $\X_\o$, and a collection of natural isomorphisms $\xi^f\colon F_A'\o\I_\o(f)\Rightarrow\I_\b(Ff)\o F_B'$, called distributors, indexed by the morphisms $f\colon A\to B$ of $\X_\o$. Finally, the $2$-morphisms of $\Indx$
\begin{align*}
&(\varphi,\varphi')\colon(F,F',\xi)\to(G,G',\kappa)\colon(\X_\o,\I_\o)\to(\X_\b,\I_\b)
\end{align*}
consist of a natural transformation $\varphi\colon F\Rightarrow G$ together with a collection $\varphi'$ of natural transformations $\varphi_A'\colon F_A'\Rightarrow G_A'\o\I_\b(\varphi_A)$ which are compatible with the distributors.
\par A tangentad in $\Indx$, a.k.a., a tangent indexed category, consists of the following data:
\begin{description}
\item[Base tangent category] A tangent category $(\X,\TT)$;

\item[Indexed category] An indexed category $\I\colon\X^\op\to\Cat$ which sends each object $M\in\X$ to a category $\X^M$ and each morphism $f\colon M\to N$ to a functor $f^\*\X^N\to\X^N$;

\item[Indexed tangent functor] A list of functors $\T'^M\colon\X^M\to\X^{\T M}$, indexed by the objects of $M$, together with a list of natural isomorphisms $\xi^f_E\colon(\T f)^\*\T'^NE\to\T'^Mf^\*E$, indexed by the morphisms $f\colon M\to N$ of $\X$;

\item[Indexed projection] A list of natural transformations $p'^M_E\colon\T'^ME\to p^\*E$;

\item[Indexed zero morphism] A list of natural transformations $z'^M_E\colon E\to z^\*\T'^ME$;

\item[Indexed sum morphism] A list of natural transformations $s'^M_E\colon\T'^M_2E\to s^\*\T'^ME$;

\item[Indexed vertical lift] A list of natural transformations $l'^M_E\colon\T'^ME\to l^\*\T'^{\T M}\T'^ME$;

\item[Indexed canonical flip] A list of natural transformations $c'^M_E\colon\T'^{\T M}\T'^ME\to c^\*\T'^{\T M}\T'^ME$.
\end{description}
The indexed projection, zero morphism, sum morphism, vertical lift, and canonical lift must also be compatible with the distributors of $\T'$ and need to satisfy the axioms of a tangent structure.
\par Finally, we mention that \cite[Theorem~5.5]{lanfranchi:grothendieck-tangent-cats} shows a $2$-equivalence between the $2$-category $\TngFib$ of tangent fibrations and the $2$-category $\TngIndx$ of tangent indexed categories.
\end{example}

\begin{example}
\label{example:tangent-restriction-categories}
Tangent restriction categories, introduced in~\cite[Definition~6.14]{cockett:tangent-cats}, are restriction categories equipped with an endofunctor $\T$ which preserves the restriction idempotents, together with structural total natural transformations $p_M\colon\T M\to M$, $z_M\colon M\to\T M$, $s_M\colon\T_2M\to\T M$, $l_M\colon\T M\to\T^2M$, and $c_M\colon\T^2M\to\T^2M$ similar to the structural natural transformations of a tangent structure and that satisfy similar axioms, but in which the $n$-fold pullback of $p$ along itself and the pullback diagram of the universality of the vertical lift are replaced with restriction pullbacks. We suggest the reader consult~\cite{cockett:restrictionI} for a definition of restriction category and~\cite{cockett:restrictionIII} for a definition of restriction limits. As noticed in~\cite[Section~4.6]{lanfranchi:tangentads-I}, tangent restriction categories are not an example of tangentads. Nevertheless, as proved in~\cite[Proposition~4.35]{lanfranchi:tangentads-I}, tangent split restriction categories, namely, tangent restriction categories whose restriction idempotents split, are precisely the tangentads in the $2$-category $\sRestrCat$ of split restriction categories, restriction functors, and total natural transformations.
\par As proved by~\cite[Lemma~4.36]{lanfranchi:tangentads-I}, every tangent restriction category embeds into a tangent split restriction category. Thus, in particular, every tangent restriction category embeds into a tangentad.
\end{example}

\subsection{The pointwise tangent structure on the Hom categories}
\label{subsection:hom-categories}
\cite[Example~2.2(vii)]{cockett:differential-bundles} shows that the category $\Cat(\X',\X)$ of functors $F\colon\X'\to\X$ comes equipped with a tangent structure defined pointwise, provided that $\X$ comes with a tangent structure $\TT$. In particular, the tangent bundle functor of a functor $F$ is the functor $\T\o F$. In this section, we generalize this construction by proving that the Hom-categories in the $2$-category of tangentads of a given $2$-category $\CC$ come with a tangent structure. This plays a crucial role in our story.
\par For starters, notice that each tangentad $(\X,\TT)$ is a tangentad in the $2$-category $\Tng(\CC)$ of tangentads of $\CC$. In particular, the tangent bundle $1$-morphism is the strong tangent morphism $(\T,c)\colon(\X,\TT)\to(\X,\TT)$ and the structural $2$-morphisms are the same of $\TT$.

\begin{proposition}
\label{proposition:hom-tangent-categories}
For two tangentads $(\X,\TT)$ and $(\X',\TT')$ (with negatives), the category of lax tangent morphisms $\Tng(\CC)[\X',\TT';\X,\TT]$ from $(\X',\TT')$ to $(\X,\TT)$ comes with a tangent structure (with negatives) defined pointwise. Concretely, the tangent bundle functor sends a lax tangent morphism $(F,\alpha)\colon(\X',\TT')\to(\X,\TT)$ to the lax tangent morphism:
\begin{align*}
&\bar\T(F,\alpha)\colon(\X',\TT')\xrightarrow{(F,\alpha)}(\X,\TT)\xrightarrow{(\T,c)}(\X,\TT)
\end{align*}
Moreover, $\bar\T$ sends a morphism $\varphi\colon(F,\alpha)\to(G,\beta)$ of lax tangent morphisms to the morphism:
\begin{align*}
&\bar\T\varphi\=\T\varphi\colon\bar\T(F,\alpha)\to\bar\T(G,\beta)
\end{align*}
Furthermore, the Hom-category $2$-functor $\Tng(\CC)(-,?)$ lifts along the forgetful functor $\U\colon\TngCat\to\Cat$ to a $2$-functor
\begin{align*}
&\Tng(\CC)^\op\times\Tng(\CC)\to\TngCat
\end{align*}
strict and contravariant in the first argument and lax and covariant in the second.
\end{proposition}
\begin{proof}
To define a tangent structure is to define a Leung functor (cf.~\cite[Definition~2.11]{lanfranchi:tangentads-I}), namely, a strong monoidal functor
\begin{align*}
&\Leung[\bar\TT]\colon\Weil\to\End\left(\Tng(\CC)[\X',\TT';\X,\TT]\right)
\end{align*}
which preserves the pointwise fundamental tangent limits. Let us define $\Leung[\bar\TT]$ as follows. $\Leung[\bar\TT]$ sends the Weil algebra $W_n$ to the endofunctor $(\T_n,c_n)\o(-)$ of $\Tng(\CC)[\X',\TT';\X,\TT]$ which postcomposes with $(\T_n,c_n)$, namely the $n$-fold pullback of the projection along itself. Moreover, the structural morphisms $p$, $z$, $s$, $l$, and $c$ (and $n$ when negatives are considered)  are sent to the tangent natural transformations $p\o(-)$, $z\o(-)$, $s\o(-)$, $l\o(-)$, and $c\o(-)$ (and $n\o(-)$, when negatives are considered), respectively. One needs to use the axioms of a tangent structure to check that these are, in fact, tangent natural transformations.
\par This defines an assignment on the generators of $\Weil$, thus it automatically extends to a strong monoidal functor $\Leung[\bar\TT]$. Let us prove that such a strong functor preserves the pointwise fundamental tangent limits. The preservation of the pointwise $n$-fold pullback of the projection is a direct consequence of $W_n$ being sent to $(\T_n,c_n)$, which corresponds to the $n$-fold pointwise pullback of the projection along itself. Therefore, the only remaining axiom to be checked is the universal property of the vertical lift. Consider a lax tangent morphism $(F,\alpha)$ and the following diagram:
\begin{equation*}
\begin{tikzcd}
{(\T_2,c_2)\o(F,\alpha)} & {(\T,c)\o(\T_2,c_2)\o(F,\alpha)} & {(\T,c)^2\o(F,\alpha)} \\
{(F,\alpha)} && {(\T,c)\o(F,\alpha)}
\arrow["{l\times z_{(\T,c)}}", from=1-1, to=1-2]
\arrow["{p\o\pi_1}"', from=1-1, to=2-1]
\arrow["{(\T,c)s}", from=1-2, to=1-3]
\arrow["{(\T,c)p}", from=1-3, to=2-3]
\arrow["z"', from=2-1, to=2-3]
\end{tikzcd}
\end{equation*}
By~\cite[Proposition~2.24]{lanfranchi:tangentads-I}, the forgetful functor $\U\colon\Tng(\CC)\to\CC$ reflects tangent limits. By applying $\U$ to the above diagram we obtain the diagram:
\begin{equation*}
\begin{tikzcd}
{\T_2\o F} & {\T\o\T_2\o F} & {\T^2\o F} \\
F && {\T\o F}
\arrow["{l\times z_\T}", from=1-1, to=1-2]
\arrow["{p\o\pi_1}"', from=1-1, to=2-1]
\arrow["{\T s}", from=1-2, to=1-3]
\arrow["{\T p}", from=1-3, to=2-3]
\arrow["z"', from=2-1, to=2-3]
\end{tikzcd}
\end{equation*}
which is a pointwise $\T$-limit by the universal property of the vertical lift of $(\X,\TT)$. Therefore, by reflectivity, the original diagram is a pointwise $(\T,c)$-limit, proving that the vertical lift is universal in $\Tng(\CC)[\X',\TT';\X,\TT]$.
\par To prove that the Hom-category $2$-functor lifts along the forgetful functor $\TngCat\to\Cat$, consider a lax tangent morphism $(F,\alpha)\colon(\X'',\TT'')\to(\X',\TT')$. It is not hard to see that the functor
\begin{align*}
&[F,\alpha\|\X,\TT]\colon[\X',\TT'\|\X,\TT]\to[\X'',\TT''\|\X,\TT]
\end{align*}
which precomposes lax tangent morphisms with $(F,\alpha)$ extends to strict tangent morphism. Similarly, given a lax tangent morphism $(G,\beta)\colon(\X,\TT)\to(\X'',\TT'')$, the lax ditributive law $\beta\colon G\o\T\Rightarrow\T''\o G$ lifts a lax distributive law $[\X',\TT'\|G,\beta]\o\bar\T\Rightarrow\bar\T\o[\X',\TT'\|G,\beta]$. Finally, tangent $2$-morphisms are sent to tangent $2$-morphisms in the obvious way.
\end{proof}

In the following, when the ambient $2$-category $\CC$ is clear from the context, we denote the Hom-tangent categories of Proposition~\ref{proposition:hom-tangent-categories} by:
\begin{align*}
&[\X',\TT'\|\X,\TT]\=\Tng(\CC)[\X',\TT';\X,\TT]
\end{align*}

The converse of Proposition~\ref{proposition:hom-tangent-categories} is also true, namely, every corepresentable $2$-functor $\Gamma\colon\CC\to\TngCat$ defines a tangentad in $\CC$.

\begin{corollary}
\label{corollary:tangentad-representable-2-functors}
A tangentad $(\X,\TT)$ in a $2$-category $\CC$ defines a corepresentable $2$-functor $[-\|\X,\TT]\colon\CC^\op\to\TngCat$. Moreover, every corepresentable $2$-functor $\CC[-\|\X]\colon\CC^\op\to\TngCat$ defines a tangentad in $\CC$.
\end{corollary}
\begin{proof}
Given a corepresentable $2$-functor $\CC[-\|\X]\colon\CC^\op\to\TngCat$, using the enriched version of the Yoneda lemma, one can show that $\X$ comes equipped with a tangent structure.
\end{proof}


\section{The formal theory of vector fields}
\label{section:vector-fields}
A vector field in differential geometry consists of a smooth assignment, for each point of a finite-dimensional smooth manifold, of a tangent vector of the tangent space at that point. In~\cite{rosicky:tangent-cats} and in~\cite{cockett:tangent-cats}, vector fields were introduced in tangent category theory and employed in~\cite{cockett:differential-equations} to introduce ordinary differential equations in tangent category theory.
\par In this section, we extend the notion of vector fields to the formal theory of tangentads. For starters, in Section~\ref{subsection:tangent-category-vector-fields} we recall the notion of vector fields in tangent category theory, the commutative monoid and the Lie algebra structures of vector fields on an object of a tangent category (with negatives), and the tangent category of vector fields as presented in~\cite{cockett:differential-equations}.
\par In Section~\ref{subsection:universal-property-vector-fields}, we characterize the universal property of vector fields in the context of tangent category theory and harness this characterization to define vector fields in the formal context of tangentads.
\par In Section~\ref{subsection:structures-vector-fields}, we prove the functoriality of the construction of vector fields, show that vector fields carry a monad structure, and recover the commutative monoid and the Lie algebra structures in the formal context.

\subsection{Vector fields in tangent category theory}
\label{subsection:tangent-category-vector-fields}
A vector field on an object $M$ in a tangent category $(\X,\TT)$ is a section $v\colon M\to\T M$ of the projection $p\colon\T M\to M$ of the tangent bundle on $M$, namely, the following diagram commutes:
\begin{equation*}
\begin{tikzcd}
M & {\T M} \\
& M
\arrow["v", from=1-1, to=1-2]
\arrow[equals, from=1-1, to=2-2]
\arrow["p", from=1-2, to=2-2]
\end{tikzcd}
\end{equation*}
Vector fields on an object $M$ in a tangent category $(\X,\TT)$ form a commutative monoid $\VF(\X,\TT;M)$ whose zero element is the zero vector field $z\colon M\to\T M$ and the sum $u+v$ of two vector fields $u,v\colon M\to\T M$ is the vector field:
\begin{align*}
&u+v\colon M\xrightarrow{\<u,v\>}\T_2M\xrightarrow{s}\T M
\end{align*}
Moreover, when the tangent category $(\X,\TT)$ has negatives, $\VF(\X,\TT;M)$ comes equipped with a Lie algebra structure whose Lie bracket $[u,v]$ between two vector fields $u,v\colon M\to\T M$ is the vector field
\begin{align*}
&[u,v]\=\{\T v\o u-c\o\T u\o v\}
\end{align*}
where
\begin{align*}
&f-g\colon X\xrightarrow{\<f,n_\T\o g\>}\T_2\T M\xrightarrow{s_{\T M}}\T^2M
\end{align*}
for two morphisms $f,g\colon X\to\T^2M$, and where $\{h\}\colon X\to\T M$ is defined by the universal property of the vertical lift as the unique morphism making the following diagram commute
\begin{equation*}
\begin{tikzcd}
X &&& {\T^2M} \\
{\T^2M} & {\T_2M} & {\T\T_2M} & {\T^2M} \\
{\T M} & M && {\T M}
\arrow["h", from=1-1, to=1-4]
\arrow["h"', from=1-1, to=2-1]
\arrow["{\left\<\{h\},p_\T\o h\right\>}"{description}, dashed, from=1-1, to=2-2]
\arrow[equals, from=1-4, to=2-4]
\arrow["{p_\T}"', from=2-1, to=3-1]
\arrow["{l\times_{\T M}z_\T}", from=2-2, to=2-3]
\arrow["{\pi_1p}"', from=2-2, to=3-2]
\arrow["{\T s}", from=2-3, to=2-4]
\arrow["{\T p}", from=2-4, to=3-4]
\arrow["p"', from=3-1, to=3-2]
\arrow[""{name=0, anchor=center, inner sep=0}, "z"', from=3-2, to=3-4]
\arrow["\lrcorner"{anchor=center, pos=0.125}, draw=none, from=2-2, to=0]
\end{tikzcd}
\end{equation*}
provided that:
\begin{align*}
&\T p\o h=z\o p\o p_\T\o h
\end{align*}
\cite[Proposition~2.10]{cockett:differential-equations} establishes that vector fields of a tangent category $(\X,\TT)$ form a new tangent category $\VF(\X,\TT)$. We briefly recall this construction:
\begin{description}
\item[Objects] An object of $\VF(\X,\TT)$ is a pair $(M,v)$ formed by an object $M$ of $\X$ together with a vector field $v\colon M\to\T M$;

\item[Morphisms] A morphism $f\colon(M,v)\to(N,u)$ of $\VF(\X,\TT)$ is a morphism $f\colon M\to N$ of $\X$ for which the following diagram commutes:
\begin{equation*}
\begin{tikzcd}
{\T M} & {\T N} \\
M & N
\arrow["{\T f}", from=1-1, to=1-2]
\arrow["v", from=2-1, to=1-1]
\arrow["f"', from=2-1, to=2-2]
\arrow["u"', from=2-2, to=1-2]
\end{tikzcd}
\end{equation*}

\item[Tangent bundle functor] The tangent bundle functor $\T^\VF$ sends a pair $(M,v)$ to $(\T M,v_\T)$, where
\begin{align*}
&v_\T\colon\T M\xrightarrow{\T v}\T^2M\xrightarrow{c}\T^2M
\end{align*}
and a morphism $f\colon(M,v)\to(N,u)$ of vector fields to the morphism $\T f\colon(\T M,v_\T)\to(\T N,u_\T)$;

\item[structural natural transformations] The structural natural transformations of $\VF(\X,\TT)$ are the same as the ones of $(\X,\TT)$.
\end{description}
$\VF$ extends to a functor, as established by the next lemma.

\begin{lemma}
\label{lemma:VF-functoriality-tangent-cats}
The assignment $\VF$ which sends a tangent category $(\X,\TT)$ to the tangent category $\VF(\X,\TT)$ of vector fields of $(\X,\TT)$ extends to a $2$-functor
\begin{align*}
&\VF\colon\TngCat\to\TngCat
\end{align*}
which sends a lax tangent morphism $(F,\alpha)\colon(\X,\TT)\to(\X',\TT')$ to the lax tangent morphism
\begin{align*}
&\VF(F,\alpha)\colon\VF(\X,\TT)\to\VF(\X',\TT')
\end{align*}
which sends a vector field $(M,v\colon M\to\T M)$ to a vector field $(FM,v_{FM})$ where
\begin{align*}
&v_{FM}\colon FM\xrightarrow{Fv}F\T M\xrightarrow{\alpha_M}\T'FM
\end{align*}
and whose distributive law coincides with $\alpha$. Moreover, $\VF(F,\alpha)$ sends a morphism $f\colon(M,v)\to(N,u)$ of vector fields to $Ff\colon(FM,v_{FM})\to(FN,u_{FN})$. Finally, given a $2$-morphism $\varphi\colon(F,\alpha)\Rightarrow(G,\beta)\colon(\X,\TT)\to(\X',\TT')$,
\begin{align*}
&\VF(\varphi)\colon\VF(F,\alpha)\Rightarrow\VF(G,\beta)
\end{align*}
is simply $\varphi$.
\end{lemma}
\begin{proof}
For starters, notice that since $p'_F\o\alpha=F p$:
\begin{align*}
&p'_{FM}\o v_{FM}=p'_{FM}\o\alpha_M\o Fv=Fp_M\o Fv=\id_{FM}
\end{align*}
Therefore, $v_{FM}$ is a vector field of $(\X',\TT')$. To prove that $Ff$ is a morphism of vector fields, consider the following diagram:
\begin{equation*}
\begin{tikzcd}
FM & {F\T M} & {\T'FM} \\
FN & {F\T N} & {\T'FN}
\arrow["Fv", from=1-1, to=1-2]
\arrow["Ff"', from=1-1, to=2-1]
\arrow["\alpha", from=1-2, to=1-3]
\arrow["{F\T f}", from=1-2, to=2-2]
\arrow["{\T'Ff}", from=1-3, to=2-3]
\arrow["Fv"', from=2-1, to=2-2]
\arrow["\alpha"', from=2-2, to=2-3]
\end{tikzcd}
\end{equation*}
The left square diagram commutes since $f$ is a morphism of vector fields, and the right square diagram commutes by naturality of $\alpha$. Therefore, $Ff$ is a morphism of vector fields. Let us now prove that $\alpha\colon(FM,v_{FM})\to(FN,u_{FN})$ is a morphism of vector fields. This amounts to prove that the following diagram commutes:
\begin{equation*}
\begin{tikzcd}
{\T'FM} & {\T'F\T M} & {{\T'}^2FM} & {{\T'}^2FM} \\
{F\T M} & {F\T^2M} & {F\T^2M} & {\T'F\T M}
\arrow["{\T'Fv}", from=1-1, to=1-2]
\arrow["{\T'\alpha}", from=1-2, to=1-3]
\arrow["{c'_{FM}}", from=1-3, to=1-4]
\arrow["{\alpha_M}", from=2-1, to=1-1]
\arrow["{F\T v}"', from=2-1, to=2-2]
\arrow["{\alpha_{\T M}}"', from=2-2, to=1-2]
\arrow["{Fc_M}"', from=2-2, to=2-3]
\arrow["{\alpha_{\T M}}"', from=2-3, to=2-4]
\arrow["{\T'\alpha_M}"', from=2-4, to=1-4]
\end{tikzcd}
\end{equation*}
However, the left square diagram commutes by naturality of $\alpha$ and the right rectangular diagram commutes by the compatibility between $\alpha$ and $c$. Finally, each tangent natural transformation $\varphi\colon(F,\beta)\Rightarrow(G,\beta)$ is sent to a tangent natural transformation $\VF(\varphi)=\varphi$.
\end{proof}

\subsection{The universal property of vector fields}
\label{subsection:universal-property-vector-fields}
This section aims to characterize the correct universal property of the tangent category $\VF(\X,\TT)$ of vector fields of a tangent category $(\X,\TT)$. For starters, consider the forgetful functor $\U\colon\VF(\X,\TT)\to(\X,\TT)$ which sends a pair $(M,v)$ to the underlying object $M$ and a morphism $f\colon(M,v)\to(N,u)$ of vector fields to the underlying morphism $f\colon M\to N$. $\U$ strictly preserves the tangent structures; therefore, it can be regarded as an object of the Hom-tangent category $[\VF(\X,\TT)\|\X,\TT]$. One can also define a natural transformation $\Univ v\colon\U\to\bar\T\U$
\begin{align*}
&\Univ v\colon\U(M,v)=M\xrightarrow{v}\T M=(\bar\T\U)(M,v)
\end{align*}
which picks out the vector field from each pair $(M,v)\in\VF(\X,\TT)$, where $\bar\T\U$ is the tangent bundle functor $\bar\T$ of $[\VF(\X,\TT)\|\X,\TT]$ applied to the object $\U$. The naturality of $\Univ v$ corresponds to the commutativity of the following diagrams
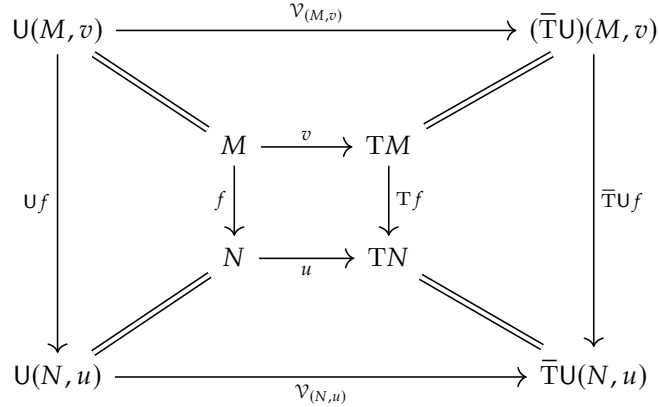
\begin{equation*}
\begin{tikzcd}
{\U(M,v)} &&& {(\bar\T\U)(M,v)} \\
& M & {\T M} \\
& N & {\T N} \\
{\U(N,u)} &&& {\bar\T\U(N,u)}
\arrow["{\Univ v_{(M,v)}}", from=1-1, to=1-4]
\arrow[equals, from=1-1, to=2-2]
\arrow["{\U f}"', from=1-1, to=4-1]
\arrow[equals, from=1-4, to=2-3]
\arrow["{\bar\T\U f}", from=1-4, to=4-4]
\arrow["v", from=2-2, to=2-3]
\arrow["f"', from=2-2, to=3-2]
\arrow["{\T f}", from=2-3, to=3-3]
\arrow["u"', from=3-2, to=3-3]
\arrow[equals, from=4-1, to=3-2]
\arrow["{\Univ v_{(N,u)}}"', from=4-1, to=4-4]
\arrow[equals, from=4-4, to=3-3]
\end{tikzcd}
\end{equation*}
for each morphism $f\colon(M,v)\to(N,u)$, which is precisely the condition a morphism $f\colon M\to N$ must satisfy in order to be a morphism of vector fields. Furthermore, $\Univ v\colon\U\Rightarrow\bar\T\U$ is compatible with the distributive laws of the two (strong) tangent morphisms $\U$ and $\T\U$. Indeed, since $c^2=\id_{\T^2}$, the following diagram commutes
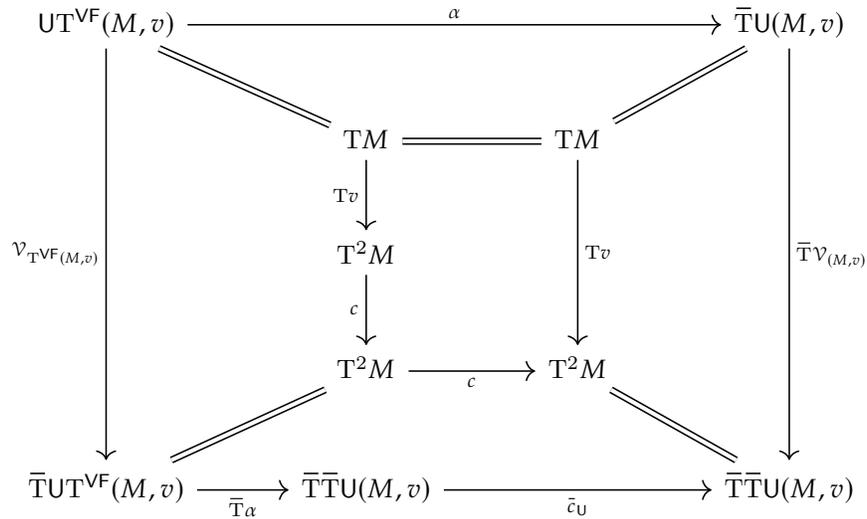
\begin{equation*}
\begin{tikzcd}
{\U\T^\VF(M,v)} &&& {\bar\T\U(M,v)} \\
& {\T M} & {\T M} \\
& {\T^2M} \\
& {\T^2M} & {\T^2M} \\
{\bar\T\U\T^\VF(M,v)} & {\bar\T\bar\T\U(M,v)} && {\bar\T\bar\T\U(M,v)}
\arrow["\alpha", from=1-1, to=1-4]
\arrow[equals, from=1-1, to=2-2]
\arrow["{\Univ v_{\T^\VF(M,v)}}"', from=1-1, to=5-1]
\arrow[equals, from=1-4, to=2-3]
\arrow["{\bar\T\Univ v_{(M,v)}}", from=1-4, to=5-4]
\arrow[equals, from=2-2, to=2-3]
\arrow["{\T v}"', from=2-2, to=3-2]
\arrow["{\T v}", from=2-3, to=4-3]
\arrow["c"', from=3-2, to=4-2]
\arrow["c"', from=4-2, to=4-3]
\arrow[equals, from=5-1, to=4-2]
\arrow["{\bar\T\alpha}"', from=5-1, to=5-2]
\arrow["{\bar c_\U}"', from=5-2, to=5-4]
\arrow[equals, from=5-4, to=4-3]
\end{tikzcd}
\end{equation*}
where $\alpha$ denotes the strict distributive law of $\U$, namely the identity. Finally, $\Univ v\colon\U\to\bar\T\U$ defines a vector field in $[\VF(\X,\TT)\|\X,\TT]$ since:
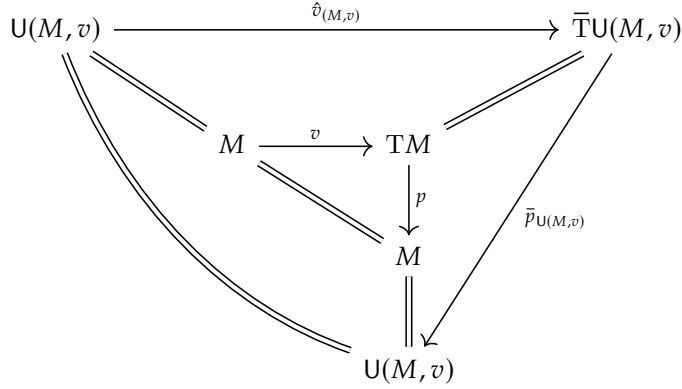
\begin{equation*}
\begin{tikzcd}
{\U(M,v)} &&& {\bar\T\U(M,v)} \\
& M & {\T M} \\
&& M \\
&& {\U(M,v)}
\arrow["{\hat v_{(M,v)}}", from=1-1, to=1-4]
\arrow[equals, from=1-1, to=2-2]
\arrow[curve={height=30pt}, equals, from=1-1, to=4-3]
\arrow[equals, from=1-4, to=2-3]
\arrow["{\bar p_{\U(M,v)}}", from=1-4, to=4-3]
\arrow["v", from=2-2, to=2-3]
\arrow[equals, from=2-2, to=3-3]
\arrow["p", from=2-3, to=3-3]
\arrow[equals, from=4-3, to=3-3]
\end{tikzcd}
\end{equation*}
We recap our findings in the following proposition.

\begin{proposition}
\label{proposition:universal-vector-field}
The forgetful functor $\U\colon\VF(\X,\TT)\to(\X,\TT)$ is a strict tangent morphism. Furthermore, the tangent natural transformation $\Univ v\colon\U\to\bar\T\U$ which picks out the vector field out of each pair $(M,v)\in\VF(\X,\TT)$ defines a vector field on $\U$ in the Hom-tangent category $[\VF(\X,\TT)\|\X,\TT]$.
\end{proposition}

Our goal is to characterize the universal property enjoyed by the vector field $\Univ v$ on $\U$ in the Hom-tangent category $[\VF(\X,\TT)\|\X,\TT]$. For starters, let us unwrap the definition of a vector field in the Hom-tangent category $[\X',\TT'\|\X'',\TT'']$ for two arbitrary tangent categories:
\begin{description}
\item[Base object] The base object of a vector field in $[\X',\TT'\|\X'',\TT'']$ consists of a lax tangent morphism $(G,\beta)\colon(\X',\TT')\to(\X'',\TT'')$;

\item[Vector field] The vector field consists of a collection
\begin{align*}
&u_A\colon GA\to\T''GA
\end{align*}
of vector fields in $(\X'',\TT'')$ for each $A\in\X'$. Such a collection must be natural in $A$, and must be compatible with the distributive law of $(G,\beta)$, namely the following diagram commutes:
\begin{equation}
\label{diagram:beta-compatibility-vector-field-u}
\begin{tikzcd}
{G\T'A} && {\T''GA} \\
{\T''G\T'A} & {{\T''}^2GA} & {{\T''}^2GA}
\arrow["{\beta_A}", from=1-1, to=1-3]
\arrow["{u_{\T'A}}"', from=1-1, to=2-1]
\arrow["{\T''u_{GA}}", from=1-3, to=2-3]
\arrow["{\beta_{\T'A}}"', from=2-1, to=2-2]
\arrow["{c''_{GA}}"', from=2-2, to=2-3]
\end{tikzcd}
\end{equation}
\end{description}
A morphism of vector fields $\varphi\colon(G,\beta;u)\to(G',\beta';u')$ is a tangent natural transformation $\varphi\colon(G,\beta)\Rightarrow(G',\beta')$, compatible with the vector fields as follows:
\begin{equation*}
\begin{tikzcd}
{\T''GA} & {\T''G'A} \\
GA & {G'A}
\arrow["{\T''\varphi_A}", from=1-1, to=1-2]
\arrow["{u_A}", from=2-1, to=1-1]
\arrow["{\varphi_A}"', from=2-1, to=2-2]
\arrow["{u'_A}"', from=2-2, to=1-2]
\end{tikzcd}
\end{equation*}
Finally, $\VF[\X',\TT'\|\X'',\TT'']$ also carries a tangent structure whose tangent bundle functor sends a vector field $(G,\beta;u)$ to $(\bar\T(G,\beta),u_{\bar\T})$, where:
\begin{align*}
&(u_{\bar\T})_A\colon\T''GA\xrightarrow{\T''u_A}{\T''}^2GA\xrightarrow{c''_{GA}}{\T''}^2GA
\end{align*}
This characterization extends readily to abstract tangentads.
\par For each tangent category $(\X',\TT')$, a vector field $u\colon(G,\beta)\to\bar\T(G,\beta)$ in the Hom-tangent category $[\X'',\TT''\|\X''',\TT''']$ induces a functor
\begin{align*}
&\Gamma_ u\colon[\X',\TT'\|\X'',\TT'']\to\VF[\X',\TT'\|\X''',\TT''']
\end{align*}
which sends a lax tangent morphism $(H,\gamma)\colon(\X',\TT')\to(\X'',\TT'')$ to the vector field
\begin{align*}
&\Gamma_ u(H,\gamma)\colon(G,\beta)\o(H,\gamma)\xrightarrow{u_{(H,\gamma)}}(\bar\T(G,\beta))\o(H,\gamma)=\bar\T((G,\beta)\o(H,\gamma))
\end{align*}
defined on the lax tangent morphism $(G,\beta)\o(H,\gamma)$ obtained by postcomposing $(H,\gamma)$ by $(G,\beta)$. Concretely, the vector field $\Gamma_ u(H,\gamma)$ is the natural transformation:
\begin{equation*}
\begin{tikzcd}
{(\X',\TT')} & {(\X'',\TT'')} & {(\X''',\TT''')}
\arrow["{(H,\gamma)}", from=1-1, to=1-2]
\arrow[""{name=0, anchor=center, inner sep=0}, "{(G,\beta)}"', bend right, from=1-2, to=1-3]
\arrow[""{name=1, anchor=center, inner sep=0}, "{\bar\T(G,\beta)}", bend left, from=1-2, to=1-3]
\arrow["u"', shorten <=5pt, shorten >=5pt, Rightarrow, from=0, to=1]
\end{tikzcd}
\end{equation*}
Furthermore, the functor $\Gamma_ u$ sends a morphism $\varphi\colon(H,\gamma)\to(H',\gamma')$ of lax tangent morphisms to a morphism of vector fields:
\begin{align*}
&\Gamma_ u(\varphi)\colon\Gamma_ u(H,\gamma)=u_{(H,\gamma)}\xrightarrow{u(\varphi)}u_{(H',\gamma')}
\end{align*}
Concretely, this corresponds to the natural transformation:
\begin{equation*}
\begin{tikzcd}
{(\X',\TT')} & {(\X'',\TT'')} & {(\X''',\TT''')}
\arrow[""{name=0, anchor=center, inner sep=0}, "{(H,\gamma)}"', bend right, from=1-1, to=1-2]
\arrow[""{name=1, anchor=center, inner sep=0}, "{(H',\gamma')}", bend left, from=1-1, to=1-2]
\arrow[""{name=2, anchor=center, inner sep=0}, "{(G,\beta)}"', bend right, from=1-2, to=1-3]
\arrow[""{name=3, anchor=center, inner sep=0}, "{\bar\T(G,\beta)}", bend left, from=1-2, to=1-3]
\arrow["\varphi"', shorten <=5pt, shorten >=5pt, Rightarrow, from=0, to=1]
\arrow["u"', shorten <=5pt, shorten >=5pt, Rightarrow, from=2, to=3]
\end{tikzcd}
\end{equation*}
The functor $\Gamma_ u$ also comes with a distributive law
\begin{align*}
&\Gamma_\beta(H,\gamma)\colon\Gamma_ u(\bar\T(H,\gamma))\to\T^\VF(\Gamma_ u(H,\gamma))
\end{align*}
so defined:
\begin{align*}
&G\T''HA\xrightarrow{\beta_{HA}}\T'''GHA
\end{align*}
To prove that $\Gamma_\beta$ is a morphism of vector fields is to show the following diagram commutes:
\begin{equation*}
\begin{tikzcd}
{\T'''G\T''HA} & {{\T'''}^2GHA} \\
& {{\T'''}^2GHA} \\
{G\T''HA} & {\T'''GHA}
\arrow["{\T'''\beta_{HA}}", from=1-1, to=1-2]
\arrow["{c'''_{GHA}}", from=2-2, to=1-2]
\arrow["{u_{\T''HA}}"', from=3-1, to=1-1]
\arrow["{\Gamma_ u(\T''(H,\gamma))_A}", bend left=50, from=3-1, to=1-1]
\arrow["{\beta_{HA}}"', from=3-1, to=3-2]
\arrow["{\T^\VF(\Gamma_ u(H,\gamma))_A}"', bend right=50, from=3-2, to=1-2]
\arrow["{\T'''u_{HA}}", from=3-2, to=2-2]
\end{tikzcd}
\end{equation*}
However, by the commutativity of Diagram~\eqref{diagram:beta-compatibility-vector-field-u} and by ${c'''}^2=\id_{{\T'''}^2}$, this diagram commutes. The compatibility of $\beta$ with the tangent structures of $(\X'',\TT'')$ and $(\X''',\TT''')$ makes $(\Gamma_ u,\Gamma_\beta)$ into a lax tangent morphism:
\begin{align*}
&\Gamma_{(G,\beta;u)}\=(\Gamma_ u,\Gamma_\beta)\colon[\X',\TT'\|\X'',\TT'']\to\VF[\X',\TT'\|\X''',\TT''']
\end{align*}
We recap our findings in the next lemma.

\begin{lemma}
\label{lemma:induced-lax-tangent-morphism-from-vector-fields}
A vector field $(G,\beta;u)$ in the Hom-tangent category $[\X'',\TT''\|\X''',\TT''']$ induces a lax tangent morphism
\begin{align*}
&\Gamma_{(G,\beta;u)}\=(\Gamma_ u,\Gamma_\beta)\colon[\X',\TT'\|\X'',\TT'']\to\VF[\X',\TT'\|\X''',\TT''']
\end{align*}
for each tangentad $(\X',\TT')$. Furthermore, $\Gamma_{(G,\beta;u)}$ is natural in $(\X',\TT')$. Concretely, given a lax tangent morphism $(F,\alpha)\colon(\X,\TT)\to(\X',\TT')$, the following diagram commutes in the $2$-category of tangent categories:
\begin{equation*}
\begin{tikzcd}
{[\X',\TT'\|\X'',\TT'']} & {\VF[\X',\TT'\|\X''',\TT''']} \\
{[\X,\TT\|\X'',\TT'']} & {\VF[\X,\TT\|\X''',\TT''']}
\arrow["{\Gamma_{(G,\beta;u)}}", from=1-1, to=1-2]
\arrow["{[F,\alpha\|\X'',\TT'']}"', from=1-1, to=2-1]
\arrow["{\VF[F,\alpha\|\X''',\TT''']}", from=1-2, to=2-2]
\arrow["{\Gamma_{(G,\beta;u)}}"', from=2-1, to=2-2]
\end{tikzcd}
\end{equation*}
Furthermore, when the base tangent morphism $(G,\beta)$ is strong (strict), the induced tangent morphism $\Gamma_{(G,\beta;u)}$ is also strong (strict).
\end{lemma}
\begin{proof}
To define the induced lax tangent morphism $\Gamma_{(G,\beta;u)}$, one can readily extend the argument used in the discussion above for tangent categories by removing the dependency of $u$ on the object $A$ and replacing natural transformations with $2$-morphisms of $\CC$. We leave the details to the reader. From the definition of the distributive law of $\Gamma_{(G,\beta;u)}$, one immediately concludes that if $(G,\beta)$ is strong, respectively strict, so is $\Gamma_{(G,\beta;u)}$. To prove that $(G,\beta;u)$ is natural in $(\X',\TT')$, consider two lax tangent morphisms $(F,\alpha)\colon(\X,\TT)\to(\X',\TT')$ and $(H,\gamma)\colon(\X',\TT')\to(\X'',\TT'')$. Thus:
\begin{align*}
&(\Gamma_{(G,\beta;u)}\o[F,\alpha\|\X',\TT'])(H,\gamma))=\Gamma_{(G,\beta;u)}((H,\gamma)\o(F,\alpha))\\
&\quad=\Gamma_{(G,\beta;u)}((H,\gamma))\o(F,\alpha)=\VF[F,\alpha\|\X''',\TT''']\o\Gamma_{(G,\beta;u)}(H,\gamma)
\end{align*}
Finally, since each $[F,\alpha\|\X'',\TT'']$ is a strict tangent morphism, $\Gamma_{(G,\beta;u)}$ is a tangent natural transformation.
\end{proof}

In particular, by Lemma~\ref{lemma:induced-lax-tangent-morphism-from-vector-fields}, the vector field $\Univ v$ of Proposition~\ref{proposition:universal-vector-field} induces a strict tangent natural transformation
\begin{align*}
&\Gamma_{(\U,\Univ v)}\colon[\X',\TT'\|\VF(\X,\TT)]\to\VF[\X',\TT'\|\X,\TT]
\end{align*}
natural in $(\X',\TT')$. In the following, we denote $\Gamma_{(\U,\Univ v)}$ simply by $\Gamma_{\Univ v}$.
\par We want to prove that $\Gamma_{\Univ v}$ is invertible. Consider another tangent category $(\X',\TT')$ together with a lax tangent morphism $(G,\beta)\colon(\X',\TT')\to(\X,\TT)$ and a vector field $u\colon(G,\beta)\to\bar\T(G,\beta)$ in the Hom-tangent category $[\X',\TT'\|\X,\TT]$. Since each $u_A\colon GA\to\T GA$ is a vector field in $(\X,\TT)$, we can define a functor
\begin{align*}
&\Lambda[u]\colon(\X',\TT')\to\VF(\X,\TT)
\end{align*}
which sends an object $A$ of $\X'$ to the vector field $u_A\colon GA\to\T GA$ and a morphism $f\colon A\to B$ of $Gf$. Notice that by naturality of $u$, $Gf\colon(GA,u_A)\to(GB,u_B)$ is a morphism of vector fields. Furthermore, $\Lambda[u]$ comes with a distributive law:
\begin{align*}
&\Lambda[\beta]\colon\Lambda[u](\T'A)=u_{\T'A}\xrightarrow{\beta}\T^\VF u_A=\T^\VF(\Lambda[u](A))
\end{align*}
Thanks to the commutativity of Diagram~\eqref{diagram:beta-compatibility-vector-field-u} and to $c^2=\id_{\T^2}$, $\beta\colon(G\T'A,u_{\T'A})\to(\T GA,c\o\T u_A)$ is a morphism of vector fields. Furthermore, since $\beta$ is compatible with the tangent structures of $(\X',\TT')$ and $(\X,\TT)$, $(\Lambda[u],\Lambda[\beta])$ becomes a lax tangent morphism.

\begin{lemma}
\label{lemma:universality-vector-fields}
Given a lax tangent morphism $(G,\beta)\colon(\X',\TT')\to(\X,\TT)$ and a vector field $u\colon(G,\beta)\to\bar\T(G,\beta)$ in the Hom-tangent category $[\X',\TT'\|\X,\TT]$, the functor $\Lambda[u]$ together with the distributive law $\Lambda[\beta]$ defines a lax tangent morphism:
\begin{align*}
&\Lambda[G,\beta;u]\=(\Lambda[u],\Lambda[\beta])\colon(\X',\TT')\to\VF(\X,\TT)
\end{align*}
\end{lemma}
\begin{proof}
In the previous discussion, we have defined $\Lambda[G,\beta;u]$ as the lax tangent morphism which picks out the vector field $u_A\colon GA\to\T GA$ for each $A\in\X'$ and whose distributive law is induced by $\beta$. In particular, the naturality of $u$ makes the functor $\Lambda[u]$ well-defined, the compatibility of $u$ with $\beta$ makes $\Lambda[\beta]$ into a morphism of vector fields, and the compatibility of $\beta$ with the tangent structures of $(\X',\TT')$ and $(\X,\TT)$, makes $\Lambda[G,\beta]$ into a lax tangent morphism.
\end{proof}

We can now prove the universal property of vector fields, which is the main result of this section.

\begin{theorem}
\label{theorem:universality-vector-fields}
The vector field $\Univ v\colon\U\to\bar\T\U$ of Proposition~\ref{proposition:universal-vector-field} is universal. Concretely, the induced strict tangent natural transformation
\begin{align*}
&\Gamma_{\Univ v}\colon[\X',\TT'\|\VF(\X,\TT)]\to\VF[\X',\TT'\|\X,\TT]
\end{align*}
makes the functor
\begin{align*}
&\TngCat^\op\xrightarrow{[-\|\X,\TT]}\TngCat^\op\xrightarrow{\VF^\op}\TngCat^\op
\end{align*}
which sends a tangent category $(\X',\TT')$ to the tangent category $\VF[\X',\TT'\|\X,\TT]$, into a corepresentable functor. In particular, $\Gamma_{\Univ v}$ is invertible.
\end{theorem}
\begin{proof}
Consider a tangent category $(\X',\TT')$ and a lax tangent morphism $(F,\alpha)\colon(\X',\TT')\to\VF(\X,\TT)$, which sends each $A\in\X'$ to the vector field $u_A\colon FA\to\T FA$. We want to compare $(F,\alpha)$ with $\Lambda[\Gamma_{\Univ v}(F,\alpha)]$. Concretely, $\Lambda[\Gamma_{\Univ v}(F,\alpha)]$ sends an object $A$ of $\X'$ to the vector fields $\Gamma_{\Univ v}(F,\alpha)_A$. However, this corresponds to the vector field $u_A\colon FA\to\T FA$. Given a morphism $\varphi\colon(F,\alpha)\to(F',\alpha')$ of $[\X',\TT'\|\VF(\X,\TT)]$
\begin{align*}
&\varphi_A\colon(FA,u_A)\to(F'A,u'_A)
\end{align*}
the morphism $\Lambda[\Gamma_{\Univ v}(\varphi)]$ coincides with $\varphi$, since:
\begin{align*}
&\Gamma_{\Univ v}(\varphi)_A=FA\xrightarrow{\varphi}F'A
\end{align*}
Conversely, consider a vector field $(G,\beta;u)$. Thus, $\Lambda[G,\beta;u]$ sends each $A$ to the vector field $u_A\colon GA\to\T GA$. Therefore, $\Gamma_{\Univ v}(\Lambda[G,\beta;u])$ coincides with $(G,\beta;u)$. Similarly, for a given morphism $\varphi\colon(G,\beta;u)\to(G',\beta';u')$ of vector fields, $\Gamma_{\Univ v}(\Lambda[\varphi])$ coincides with $\varphi$.
\par Finally, it is not hard to see that the functor
\begin{align*}
&\Lambda\colon\VF[\X',\TT'\|\X,\TT]\to[\X',\TT'\|\VF(\X,\TT)]
\end{align*}
extends to a strict tangent morphism, since $\Lambda[\T^\VF(G,\beta;u)]$ sends $A\in\X'$ to the vector field
\begin{align*}
&\T GA\xrightarrow{\T u_A}\T^2GA\xrightarrow{c_{GA}}\T^2GA
\end{align*}
while $\bar\T(\Lambda[G,\beta;u])$ sends $A\in\X'$ to the vector field
\begin{align*}
&\T^\VF(u_A)\colon\T GA\xrightarrow{\T u_A}\T^2GA\xrightarrow{c_{GA}}\T^2GA
\end{align*}
Therefore, $\Lambda$ inverts $\Gamma_{\Univ v}$.
\end{proof}

Theorem~\ref{theorem:universality-vector-fields} establishes the correct universal property of the vector fields construction. Thanks to this result, we can now introduce the notion of vector fields in the formal context of tangentads.

\begin{definition}
\label{definition:construction-vector-fields}
A tangentad $(\X,\TT)$ in a $2$-category $\CC$ \textbf{admits the construction of vector fields} if there exists a tangentad $\VF(\X,\TT)$ of $\CC$ together with a strict tangent morphism $\U\colon\VF(\X,\TT)\to(\X,\TT)$ and a vector field $\Univ v\colon\U\to\bar\T\U$ in the Hom-tangent category $[\VF(\X,\TT)\|\X,\TT]$ such that the induced tangent natural transformation $\Gamma_{\Univ v}$ of Lemma~\ref{lemma:induced-lax-tangent-morphism-from-vector-fields} is invertible. The vector field $\Univ v$ is called the \textbf{universal vector field} of $(\X,\TT)$ and $\VF(\X,\TT)$ is called the \textbf{tangentad of vector fields} of $(\X,\TT)$.
\end{definition}

\begin{definition}
\label{definition:construction-vector-fields-2-category}
A $2$-category $\CC$ \textbf{admits the construction of vector fields} provided that every tangentad of $\CC$ admits the construction of vector fields.
\end{definition}

We can now rephrase Theorem~\ref{theorem:universality-vector-fields} as follows.

\begin{corollary}
\label{corollary:universality-vector-fields}
The $2$-category $\Cat$ of categories admits the construction of vector fields and the tangentad of vector fields of a tangentad $(\X,\TT)$ of $\Cat$ is the tangent category $\VF(\X,\TT)$ of vector fields of $(\X,\TT)$.
\end{corollary}

To ensure that Definitions~\ref{definition:construction-vector-fields} and~\ref{definition:construction-vector-fields-2-category} are well-posed, one requires the tangentad of vector fields of a given tangentad to be unique. The next proposition establishes that such a construction is defined uniquely up to a unique isomorphism.

\begin{proposition}
\label{proposition:uniqueness-vector-fields}
If a tangentad $(\X,\TT)$ admits the construction of vector fields, the tangentad of vector fields $\VF(\X,\TT)$ of $(\X,\TT)$ is unique up to a unique isomorphism which extends to an isomorphism of the corresponding universal vector fields of $(\X,\TT)$.
\end{proposition}
\begin{proof}
Let $\U'\colon\VF'(\X,\TT)\to(\X,\TT)$ together with a vector field $\Univ{v'}\colon\U'\to\bar\T\U'$ be a second universal vector field of $(\X,\TT)$. Via the isomorphism
\begin{align*}
&[\VF(\X,\TT)\|\VF(\X,\TT)]\cong\VF[\VF(\X,\TT)\|\X,\TT]\cong[\VF(\X,\TT)\|\VF'(\X,\TT)]
\end{align*}
the identity on $\VF(\X,\TT)$, induces a tangent morphism $\VF(\X,\TT)\to\VF'(\X,\TT)$ which inverts the tangent morphism $\VF'(\X,\TT)\to\VF(\X,\TT)$ induced by the identity on $\VF'(\X,\TT)$ via the isomorphism:
\begin{align*}
&[\VF'(\X,\TT)\|\VF'(\X,\TT)]\cong\VF[\VF'(\X,\TT)\|\X,\TT]\cong[\VF'(\X,\TT)\|\VF(\X,\TT)]\qedhere
\end{align*}
\end{proof}

\subsection{The formal structures of vector fields}
\label{subsection:structures-vector-fields}
In Section~\ref{subsection:tangent-category-vector-fields}, we recalled that (a) vector fields of a tangent category $(\X,\TT)$ define a new tangent category $\VF(\X,\TT)$; (b) the assigment which sends a tangent category $(\X,\TT)$ to $\VF(\X,\TT)$ is $2$-functorial; (c) vector fields over an object $M$ in a tangent category form a commutative monoid $(\VF(X,\TT;M),z_M,+)$; (d) vector fields over an object $M$ in a tangent category with negatives form an abelian group $(\VF(X,\TT;M),z_M,+,-)$; (e) vector fields over an object $M$ in a tangent category with negatives form a Lie algebra $(\VF(X,\TT;M),z_M,+,[,])$.
\par In this section, we recover the same structures for vector fields in the context of tangentads. For starters, notice that, in the context of tangentads, (a) is a built-in property, since $\VF(\X,\TT)$ is, by definition, a tangentad.
\par In this section, $\CC$ denotes a $2$-category which admits the construction of vector fields.

\subsubsection*{The functoriality of the construction of vector fields}
\label{subsubsection:functoriality-vector-fields}
The first step is to show that (b) $\VF$ is $2$-functorial. Consider a lax tangent morphism $(F,\alpha)\colon(\X,\TT)\to(\X',\TT')$, which induces a functor:
\begin{align*}
&[\VF(\X,\TT)\|F,\alpha]\colon[\VF(\X,\TT)\|\X,\TT]\to[\VF(\X,\TT)\|\X',\TT']
\end{align*}
Since, by Lemma~\ref{lemma:VF-functoriality-tangent-cats}, $\VF$ is functorial in the context of tangent categories, this also defines a lax tangent morphism:
\begin{align*}
&\VF[\VF(\X,\TT)\|F,\alpha]\colon\VF[\VF(\X,\TT)\|\X,\TT]\to\VF[\VF(\X,\TT)\|\X',\TT']
\end{align*}
However, by the universal property of vector fields:
\begin{align*}
&\VF[\VF(\X,\TT)\|\X',\TT']\cong[\VF(\X,\TT)\|\VF(\X',\TT')]
\end{align*}
Therefore, we obtain a functor
\begin{align*}
&\VF[\VF(\X,\TT)\|\X,\TT]\to\VF[\VF(\X,\TT)\|\X',\TT']\cong[\VF(\X,\TT)\|\VF(\X',\TT')]
\end{align*}
which sends the universal vector field $\Univ v$ of $(\X,\TT)$ to a lax tangent morphism:
\begin{align*}
&\VF(F,\alpha)\colon\VF(\X,\TT)\to\VF(\X',\TT')
\end{align*}
Consider now a tangent $2$-morphism $\varphi\colon(F,\alpha)\Rightarrow(G,\beta)$, which induces a tangent natural transformation:
\begin{align*}
&[\VF(\X,\TT)\|\varphi]\colon[\VF(\X,\TT)\|F,\alpha]\Rightarrow[\VF(\X,\TT)\|G,\beta]
\end{align*}
Since, by Lemma~\ref{lemma:VF-functoriality-tangent-cats}, $\VF$ is $2$-functorial in the context of tangent categories, this also defines a natural transformation:
\begin{align*}
&\VF[\VF(\X,\TT)\|\varphi]\colon\VF[\VF(\X,\TT)\|F,\alpha]\Rightarrow\VF[\VF(\X,\TT)\|G,\beta]
\end{align*}
Therefore, under the isomorphism
\begin{align*}
&\VF[\VF(\X,\TT)\|\X',\TT']\cong[\VF(\X,\TT)\|\VF(\X',\TT')]
\end{align*}
the morphism $\VF[\VF(\X,\TT)\|\varphi]_{\Univ v}$ associated to the universal vector field $\Univ v$ corresponds to a tangent $2$-morphism:
\begin{align*}
&\VF(\varphi)\colon\VF(F,\alpha)\Rightarrow\VF(G,\beta)
\end{align*}

\begin{proposition}
\label{proposition:VF-functoriality}
The assignment $\VF$ which sends a tangentad $(\X,\TT)$ to the tangentad of vector fields $\VF(\X,\TT)$ of $(\X,\TT)$ extends to a $2$-functor:
\begin{align*}
&\VF\colon\Tng(\CC)\to\Tng(\CC)
\end{align*}
\end{proposition}
\begin{proof}
In the previous discussion, we defined the application of $\VF$ on a lax tangent morphism\linebreak $(F,\alpha)\colon(\X,\TT)\to(\X',\TT')$ as the lax tangent morphism corresponding to the universal vector field $\Univ v$ of $(\X,\TT)$ along the functor
\begin{align*}
&\VF[\VF(\X,\TT)\|\X,\TT]\xrightarrow{\VF[\VF(\X,\TT)\|F,\alpha]}\VF[\VF(\X,\TT)\|\X',\TT']\cong[\VF(\X,\TT)\|\VF(\X',\TT')]
\end{align*}
and the application of $\VF$ on a tangent $2$-morphism $\varphi\colon(F,\alpha)\Rightarrow(G,\beta)$ as the tangent $2$-morphism corresponding to the universal vector field $\Univ v$ of $(\X,\TT)$ along the natural transformation
\begin{align*}
&\VF[\VF(\X,\TT)\|\varphi]\colon\VF[\VF(\X,\TT)\|F,\alpha]\Rightarrow\VF[\VF(\X,\TT)\|G,\beta]
\end{align*}
under the isomorphism:
\begin{align*}
&\VF[\VF(\X,\TT)\|\X',\TT']\cong[\VF(\X,\TT)\|\VF(\X',\TT')]
\end{align*}
To prove the functoriality of $\VF$, for starters, consider the application of $\VF$ on the identity on $(\X,\TT)\to(\X,\TT)$. The induced functor $\VF[\VF(\X,\TT)\|\id_{(\X,\TT)}]$ coincides with the identity. Consequently, the lax tangent morphism $\VF(\id_{(\X,\TT)})$ corresponding to the universal vector field must be identity on $\VF(\X,\TT)$. Consider now two lax tangent morphisms $(F,\alpha)\colon(\X,\TT)\to(\X',\TT')$ and $(G,\beta)\colon(\X',\TT')\to(\X'',\TT'')$. Thus, $\VF((G,\beta)\o(F,\alpha))$ is the lax tangent morphism corresponding to he universal vector field via the functor:
\begin{align*}
&\VF[\VF(\X,\TT)\|\X,\TT]\xrightarrow{\VF[\VF(\X,\TT)\|(G,\beta)\o(F,\alpha)]}\VF[\VF(\X,\TT)\|\X'',\TT'']\cong[\VF(\X,\TT)\|\VF(\X'',\TT'')]
\end{align*}
However, by the functoriality of $\VF[\VF(\X,\TT)\|-]$:
\begin{align*}
&\VF[\VF(\X,\TT)\|(G,\beta)\o(F,\alpha)]=\VF[\VF(\X,\TT)\|G,\beta]\o\VF[\VF(\X,\TT)\|F,\alpha]
\end{align*}
Thus, we can split the functor as follows:
\begin{align*}
&\VF[\VF(\X,\TT)\|\X,\TT]\xrightarrow{\VF[\VF(\X,\TT)\|G,\beta]}\VF[\VF(\X,\TT)\|\X',\TT']\xrightarrow{\VF[\VF(\X,\TT)\|F,\alpha]}\VF[\VF(\X,\TT)\|\X'',\TT'']\\
&\quad\cong[\VF(\X,\TT)\|\VF(\X'',\TT'')]
\end{align*}
In particular, this functor sends $\Univ v$ to $\VF(G,\beta)\o\VF(F,\alpha)$, thus:
\begin{align*}
&\VF((G,\beta)\o(F,\alpha))=\VF(G,\beta)\o\VF(F,\alpha)
\end{align*}
Using a similar argument, the identity over a lax tangent morphism $(F,\alpha)$ must correspond to the identity over $\VF(F,\alpha)$ and the composition $\psi\o\varphi$ of two tangent $2$-morphisms correspond to $\VF(\psi)\o\VF(\varphi)$.
\end{proof}

\subsubsection*{The commutative monoid of vector fields}
\label{subsubsection:commutative-monoid-vector-fields}
This section is concerned with the problem of reconstructing (c) the commutative monoid structure on the set of vector fields over an object in the formal context of tangentads. For this purpose, we assume the existence of the $n$-fold $2$-pullback
\begin{equation*}
\begin{tikzcd}
{\VF_n(\X,\TT)} & {\VF(\X,\TT)} \\
{\VF(\X,\TT)} & {(\X,\TT)}
\arrow["{\pi_n}", from=1-1, to=1-2]
\arrow["{\pi_1}"', from=1-1, to=2-1]
\arrow["\lrcorner"{anchor=center, pos=0.125}, draw=none, from=1-1, to=2-2]
\arrow["\U", from=1-2, to=2-2]
\arrow["\dots"{marking, allow upside down}, shift left=3, draw=none, from=2-1, to=1-2]
\arrow["\U"', from=2-1, to=2-2]
\end{tikzcd}
\end{equation*}
in $\CC$, of $\U\colon\VF(\X,\TT)\to(\X,\TT)$ along itself for each tangentad $(\X,\TT)$ of $\CC$ and for each positive integer $n$. Notice that, by Proposition~\ref{proposition:pullbacks} $\VF_n(\X,\TT)$ is a $2$-pullback in $\Tng(\CC)$.
\par In the context of tangent category theory, $\VF_n(\X,\TT)$ is the tangent category of tuples $(M;v_1\,v_n)$ formed by an object $M$ of $\X$ together with $n$ vector fields $v_1\,v_n\colon M\to\T M$ of $M$. A morphism $f\colon(M;v_1\,v_n)\to(N;u_1\,u_n)$ in $\VF_n(\X,\TT)$ is a morphism $f\colon M\to N$ which is a morphism of vector fields $f\colon(M,v_k)\to(N,u_k)$ for each $k=1\,n$. The tangent structure of $\VF_n(\X,\TT)$ is defined as in $\VF(\X,\TT)$ with the tangent bundle functor acting on each vector field of the $n$-tuples. The $k$-th projection $\pi_k\colon\VF_n(\X,\TT)\to\VF(\X,\TT)$ is the strict tangent morphism which projects each tuple $(M;v_1\,v_n)$ into the $k$-th vector field $(M,v_k)$. 
\par For starters, let us define the zero element. Notice first, that the zero morphism $z\colon\id_\X\Rightarrow\T$ of a tangentad $(\X,\TT)$ defines a vector field $z\colon\id_{(\X,\TT)}\to(\T,c)\o\id_{(\X,\TT)}$ on the identity of $(\X,\TT)$ in the Hom-tangent category $[\X,\TT\|\X,\TT]$, namely:
\begin{align*}
&z_{(\X,\TT)}\in\VF[\X,\TT\|\X,\TT]
\end{align*}
By the universal property of vector fields, this corresponds to a tangent morphism:
\begin{align*}
&0\colon(\X,\TT)\to\VF(\X,\TT)
\end{align*}
In the context of tangent category theory, $0$ sends an object $M$ of $(\X,\TT)$ to the zero vector field\linebreak$(M,z_M\colon M\to\T M)$ on $M$ and a morphism $f\colon M\to N$ to the morphism $f\colon(M,z_M)\to(N,z_n)$, which is a morphism of vector fields since $z$ is natural.

\begin{lemma}
\label{lemma:0-is-strict}
The tangent morphism $0\colon(\X,\TT)\to\VF(\X,\TT)$ is strict.
\end{lemma}
\begin{proof}
The strict tangent morphism $\Gamma_{\Univ v}\colon[\X,\TT\|\VF(\X,\TT)]\to\VF[\X,\TT\|\X,\TT]$ induced by the universal vector field $\Univ v$ sends $0\o(\T,c)$ to $z_{(\T,c)}$ which, by the compatibility between the zero morphism and the canonical flip is equal to the vector field $c\o\bar\T z$, which is also equal to the vector field $c\o\bar\T(\Gamma_{\Univ v}(0))$. Since $\U$ is a strict tangent morphism and that $\Univ v$, regarded a $2$-morphism, is a tangent morphism:
\begin{align*}
&c\o\bar\T(\Gamma_{\Univ v}(0))=\Univ v_{(\T,c)}(0)=\Gamma_{\Univ v}((\T^\VF,c^\VF)\o0)
\end{align*}
By the universal property of $\Univ v$, we conclude that $0\o(\T,c)=(\T^\VF,c^\VF)\o0$. One can apply a similar argument to the structural natural transformations to prove that $0$ is a strict tangent morphism.
\end{proof}

The next step is to define the sum of two vector fields on the same base. For starters, notice that the tangentad $\VF_2(\X,\TT)$ captures precisely the abstract version of the tangent category of pairs of vector fields on the same base object, therefore, we expect the sum to be a tangent morphism of type $+\colon\VF_2(\X,\TT)\to\VF(\X,\TT)$.
\par Consider first the general case where $n$ is any positive integer and let $\pi_k\colon\VF_n(\X,\TT)\to\VF(\X,\TT)$ be the $k$-th projection. However, $\pi_k$ is a (strict) tangent morphism, so in particular, an object of the Hom-tangent category $[\VF_n(\X,\TT)\|\VF(\X,\TT)]$. By the universal property of vector fields, this corresponds to a vector field 
\begin{align*}
\Univ v_k&\colon\pi_k\o\U\to\bar\T(\pi_k\o\U)
\end{align*}
of $[\VF_n(\X,\TT)\|\X,\TT]$.
\par In the context of tangent category theory, $\Univ v_k$ is the natural transformation
\begin{align*}
&\Univ v_k(M;v_1\,v_n)\colon\U\pi_k(M;v_1\,v_n)=M\xrightarrow{v_k}\T M=\bar\T\U\pi_k(M;v_1\,v_n)
\end{align*}
which picks out the $k$-th vector field in each tuple $(M;v_1\,v_n)\in\VF_n(\X,\TT)$.
\par Consider now the case $n=2$. The projections $\pi_1,\pi_2\colon\VF_2(\X,\TT)\to\VF(\X,\TT)$ induce two vector fields $\Univ v_1$ and $\Univ v_2$ over the same base object $\U_2\=\pi_1\o\U=\pi_2\o\U$. Thus, since these are vector fields in a tangent category, we can define a new vector field $\Univ v_1+\Univ v_2$, using that vector fields on the same base object of a tangent category can be summed together. By the universal property of vector fields, $\Univ v_1+\Univ v_2$ must correspond to a tangent morphism:
\begin{align*}
+\colon\VF_2(\X,\TT)\to\VF(\X,\TT)
\end{align*}
In the context of tangent category theory, $+$ sends a tuple $(M;u,v)$ formed by an object $M$ of $(\X,\TT)$ and by two vector fields $u,v\colon M\to\T M$ of $M$ to the vector field $(M,u+v)$.

\begin{lemma}
\label{lemma:+-is-strict}
The tangent morphism $+\colon\VF_2(\X,\TT)\to\VF(\X,\TT)$ is strict.
\end{lemma}
\begin{proof}
The strict tangent morphism $\Gamma_{\Univ v}\colon[\VF_2(\X,\TT)\|\VF(\X,\TT)]\to\VF[\VF_2(\X,\TT)\|\X,\TT]$ induced by the universal vector field $\Univ v$ sends $+\o(\T^{\VF_2},c^{\VF_2})$ to:
\begin{align*}
&(\Univ v_1+\Univ v_2)_{(\T,c)}=(\Univ v_1)_{(\T,c)}+(\Univ v_2)_{(\T,c)}
\end{align*}
However, $\Univ v_k=\Gamma_{\Univ v}(\pi_k)$. Using that $\U$ is a strict tangent morphism and that $\Univ v$ is a tangent $2$-morphism:
\begin{align*}
&\quad(\Univ v_1)_{(\T,c)}+(\Univ v_2)_{(\T,c)}\\
=&\quad\Gamma_{\Univ v}(\pi_1)_{(\T,c)}+\Gamma_{\Univ v}(\pi_2)_{(\T,c)}\\
=&\quad\bar\T(\Gamma_{\Univ v}(\pi_1))+\bar\T(\Gamma_{\Univ v}(\pi_2))\\
=&\quad\bar\T(\Univ v_1+\Univ v_2)\\
=&\quad(\T^\VF,c^\VF)\o+
\end{align*}
Similarly, one can prove the compatibility with the structural $2$-morphisms.
\end{proof}

\begin{theorem}
\label{theorem:commutative-monoid-vector-fields}
Consider a $2$-category $\CC$ which admits the construction of vector fields. Let us also assume the existence of $n$-fold $2$-pullbacks of $\U\colon\VF(\X,\TT)\to(\X,\TT)$ along itself in $\CC$. Therefore, $\U$ equipped with the strict tangent morphisms $0\colon(\X,\TT)\to\VF(\X,\TT)$ and $+\colon\VF_2(\X,\TT)\to\VF(\X,\TT)$ becomes an additive bundle in $\Tng_=(\CC)$.
\end{theorem}
\begin{proof}
By the universal property of vector fields, the conditions required for $(\U,0,+)$ to be an additive bundle are equivalent to the following equations of vector fields in the correct Hom-tangent categories:
\begin{align*}
&z+\Univ v=\Univ v\\
&\Univ v_1+\Univ v_2=\Univ v_2+\Univ v_1\\
&(\Univ v_1+\Univ v_2)+\Univ v_3=\Univ v_1+(\Univ v_2+\Univ v_3)
\end{align*}
In particular, the first equation establishes unitality, the second equation ensures commutativity, and the third equation guarantees associativity.
\end{proof}

\subsubsection*{The abelian group of vector fields}
\label{subsubsection:abelian-group-vector-fields}
In this section, we show that (d) when a base tangentad $(\X,\TT)$ admits negatives, also the tangentad $\VF(\X,\TT)$ of vector field admits negatives. Moreover, we use negatives to define a negation morphism $-\colon\VF(\X,\TT)\to\VF(\X,\TT)$ which makes the additive bundle $(\U,0,+)$ of Theorem~\ref{theorem:commutative-monoid-vector-fields} into an abelian group bundle.

\begin{proposition}
\label{proposition:negatives-vector-fields}
If a tangentad $(\X,\TT)$ admits negatives and the construction of vector fields, the tangentad $\VF(\X,\TT)$ of vector fields of $(\X,\TT)$ admits negatives as well.
\end{proposition}
\begin{proof}
Consider a tangentad $(\X,\TT)$ with negatives. By Proposition~\ref{proposition:hom-tangent-categories}, also each Hom-tangent category $[\X',\TT'\|\X,\TT]$ admits negatives. In particular, the tangent category $[\VF(\X,\TT)\|\X,\TT]$ admits negatives, with negation denoted by $\bar n$. It is not hard to see that this negation defines a morphism of vector fields $\bar n\colon\Gamma_{\Univ v}(\T^\VF,c^\VF)\to\Gamma_{\Univ v}(\T^\VF,c^\VF)$. By the universal property of vector fields, this induces a tangent $2$-morphism $n\colon(\T^\VF,c^\VF)\Rightarrow(\T^\VF,c^\VF)$. The underlying $2$-morphism $n\colon\T^\VF\Rightarrow\T^\VF$ defines a negation for the tangentad $\VF(\X,\TT)$.
\end{proof}

The additive structure on $\U\colon\VF(\X,\TT)\to(\X,\TT)$ of Theorem~\ref{theorem:commutative-monoid-vector-fields} is defined by the commutative monoid structure on vector fields in the Hom-tangent categories and by the universal property of vector fields. When the base tangentad $(\X,\TT)$ admits negatives, by Proposition~\ref{proposition:hom-tangent-categories}, so does $\VF[\X',\TT'\|\X,\TT]$. In particular, $\VF[\VF(\X,\TT)\|\X,\TT]$ admits negatives. Therefore, one can take the negation of the universal vector field and defining a new vector field $-\Univ v$. By the universal property of vector fields, $-\Univ v$ corresponds to a tangent morphism:
\begin{align*}
-\colon\VF(\X,\TT)\to\VF(\X,\TT)
\end{align*}
In the context of tangent category theory, $-$ sends a vector field $(M,v)$ to the vector field $(M,-v)$.

\begin{lemma}
\label{lemma:--is-strict}
The tangent morphism $-\colon\VF(\X,\TT)\to\VF(\X,\TT)$ is strict.
\end{lemma}
\begin{proof}
The strict tangent morphism $\Gamma_{\Univ v}\colon[\VF(\X,\TT)\|\VF(\X,\TT)]\to\VF[\VF(\X,\TT)\|\X,\TT]$ induced by the universal vector field $\Univ v$ sends $-\o(\T^\VF,c^\VF)$ to:
\begin{align*}
&(-\Univ v)_{(\T,c)}=\T^\VF(-\Univ v)=\Gamma_{\Univ v}((\T^\VF,c^\VF)\o-)
\end{align*}
By the universal property of vector fields, $-\o(\T^\VF,c^\VF)=(\T^\VF,c^\VF)\o-$. Similarly, one can prove the compatibility with the structural $2$-morphisms.
\end{proof}

\begin{theorem}
\label{theorem:abelian-group-vector-fields}
Consider a $2$-category $\CC$ which admits the construction of vector fields. Let us also assume the existence of $n$-fold $2$-pullbacks of $\U\colon\VF(\X,\TT)\to(\X,\TT)$ along itself in $\CC$. For each tangentad $(\X,\TT)$ which admits negatives, the additive bundle $(\U\colon\VF(\X,\TT),0,+)$ of Theorem~\ref{theorem:commutative-monoid-vector-fields} comes with a negation $-\colon\VF(\X,\TT)\to\VF(\X,\TT)$ which makes $\U$ into an abelian group bundle in $\Tng_=\CC$.
\end{theorem}
\begin{proof}
By the universal property of vector fields, the condition required for $(\U,0,+,-)$ to be an abelian group bundle is equivalent to the following equation of vector fields:
\begin{align*}
&-\Univ v+\Univ v=z
\end{align*}
In particular, this establishes that the negation defines inverses for the sum.
\end{proof}

\subsubsection*{The Lie algebra of vector fields}
\label{subsubsection:lie-algebra-vector-fields}
In this section, we show (e) how to equip the abelian group bundle of Theorem~\ref{theorem:abelian-group-vector-fields} with a suitable Lie bracket, provided the existence of negatives. In Section~\ref{subsubsection:commutative-monoid-vector-fields}, we introduced the vector fields $\Univ v_k$ induced by the projections $\pi_k\colon\VF_n(\X,\TT)\to(\X,\TT)$. When $(\X,\TT)$ admits negatives, so does the Hom-tangent category $[\VF_2(\X,\TT),(\X,\TT)]$, therefore, we can take the Lie bracket of the two vector fields $\Univ v_1$ and $\Univ v_2$ over $\U_2=\U\o\pi_1=\U\o\pi_2$ and define a new vector field $[\Univ v_1,\Univ v_2]$. By the universal property of vector fields, this corresponds to a tangent morphism:
\begin{align*}
&[,]\colon\VF_2(\X,\TT)\to\VF(\X,\TT)
\end{align*}
In the context of tangent category theory $[,]$ sends a tuple $(M;u,v)$ formed by an object $M$ of $(\X,\TT)$ and by two vector fields $u,v\colon M\to\T M$ of $M$ to the vector field $(M,[u,v])$.

\begin{remark}
\label{remark:functoriality-Lie-bracket}
In the context of tangent category theory $[,]$ is, by construction, a functor. So, in particular, a morphism $f\colon(M;v_1,v_2)\to(N;u_1,u_2)$ of $\VF_2(\X,\TT)$ is sent to a morphism:
\begin{align*}
[,](f)&\colon(M,[v_1,v_2])\to(N,[u_1,u_2])
\end{align*}
However, a morphism of $\VF_2(\X,\TT)$ is a morphism $f\colon M\to N$ which is a morphism of vector fields $f\colon(M,v_1)\to(M,u_1)$ and $f\colon(M,v_2)\to(N,u_2)$. $[,]$ sends a morphism $f$ to itself. Therefore, the functoriality of $[,]$ is equivalent to the following statement in differential geometry: if two pairs of vector fields $v_1,u_1$ and $v_2,u_2$ are $f$-related, namely, $f$ is a morphism of vector fields $f\colon(M,v_1)\to(N,u_1)$ and $f\colon(M,v_2)\to(N,u_2)$, the vector fields $[v_1,v_2]$ and $[u_1,u_2]$ are also $f$-related. This was proven in tangent category theory in~\cite[Proposition~2.9]{cockett:differential-equations}. The novelty provided by the formal approach is in the simplicity of the proof. While in the original proof this was shown explicitly, from the formal point of view, this is a direct consequence of $[,]$ being a functor. This shows that the formal approach can also be adopted to prove results in tangent category theory.
\end{remark}

\begin{lemma}
\label{lemma:lie-bracket-is-strict}
The tangent morphism $[,]\colon\VF_2(\X,\TT)\to\VF(\X,\TT)$ is strict.
\end{lemma}
\begin{proof}
We leave to the reader to complete the details of this proof, since it is substantially the same proof as the ones of Lemmas~\ref{lemma:0-is-strict},~\ref{lemma:+-is-strict}, and~\ref{lemma:+-is-strict}.
\end{proof}

In order to state the main result of this section, first we need to introduce the notion of a Lie algebra bundle in a category $\X$.

\begin{definition}
\label{definition:lie-algebra-bundle}
A \textbf{Lie algebra bundle} in a category $\X$ consists of an additive bundle $\q\colon E\to B$ together with a morphism $[,]_q\colon E_2\to E$ which satisfies the following conditions
\begin{equation*}
\begin{tikzcd}
{E_3} & {E_2} \\
{E_2} & E
\arrow["{s_q\times_B\id_E}", from=1-1, to=1-2]
\arrow["{[,]_q\times_B\id_E}"', from=1-1, to=2-1]
\arrow["{[,]_q}", from=1-2, to=2-2]
\arrow["{s_q}"', from=2-1, to=2-2]
\end{tikzcd}\hfill\quad
\begin{tikzcd}
E & {E_2} \\
B & E
\arrow["{\<z_q\o q,\id_E\>}", from=1-1, to=1-2]
\arrow["q"', from=1-1, to=2-1]
\arrow["{[,]_q}", from=1-2, to=2-2]
\arrow["{z_q}"', from=2-1, to=2-2]
\end{tikzcd}\hfill\quad
\begin{tikzcd}
{E_2} & {E_2} \\
B & E
\arrow["{\<[,]_q,[,]_q\o\tau\>}", from=1-1, to=1-2]
\arrow["q"', from=1-1, to=2-1]
\arrow["{s_q}", from=1-2, to=2-2]
\arrow["{z_q}"', from=2-1, to=2-2]
\end{tikzcd}
\end{equation*}
\begin{equation*}
\begin{tikzcd}
{E_3} &&& {E_3} & {E_2} \\
B &&&& E
\arrow["{\<[[,]_q,]_q,[[,]_q,]_q\.\sigma,[[,]_q,]_q\sigma^2\>}", from=1-1, to=1-4]
\arrow["q"', from=1-1, to=2-1]
\arrow["{s_q\times_B\id_E}", from=1-4, to=1-5]
\arrow["{s_q}", from=1-5, to=2-5]
\arrow["{z_q}"', from=2-1, to=2-5]
\end{tikzcd}
\end{equation*}
where $[[,]_q,]_q$ denotes
\begin{align*}
&[[,]_q,]_q\colon E_3\xrightarrow{[,]_q\times_B\id_E}E_2\xrightarrow{[,]_q}E
\end{align*}
and $\sigma$ denotes the ciclic permutation $(1\quad 2\quad 3)$ whose action is induced by the canonical symmetry $\tau\colon E\times_BE'\to E'\times_BE$.
\end{definition}

The first two diagrams in Definition~\ref{definition:lie-algebra-bundle} represent left additivity of the Lie bracket:
\begin{align*}
&[x+_qy,z]_q=[x,z]_q+_q[y,z]_q\\
&[0_q,x]_q=0_q
\end{align*}
The third diagram represents antisymmetry:
\begin{align*}
&[x,y]_q+[y,x]_q=0_q
\end{align*}
Finally, the last diagram represents the Jacobi identity:
\begin{align*}
&[[x,y]_q,z]_q+[[z,x]_qy]_q+[[y,z]_q,x]_q=0_q
\end{align*}

\begin{theorem}
\label{theorem:Lie-algebra-vector-fields}
Consider a $2$-category $\CC$ which admits the construction of vector fields. Let us also assume the existence of $n$-fold $2$-pullbacks of $\U\colon\VF(\X,\TT)\to(\X,\TT)$ along itself in $\CC$. For each tangentad $(\X,\TT)$ which admits negatives, the abelian group bundle $(\U\colon\VF(\X,\TT),0,+,-)$ of Theorem~\ref{theorem:abelian-group-vector-fields} comes with a Lie bracket $[,]\colon\VF_2(\X,\TT)\to\VF(\X,\TT)$ which makes $\U$ into a Lie algebra bundle in $\Tng_=\CC$.
\end{theorem}
\begin{proof}
By the universal property of vector fields, the conditions required for $(\U,0,+,-,[,])$ to be a Lie algebra bundle are equivalent to the following equations of vector fields
\begin{align*}
&[\Univ v_1+\Univ v_2,\Univ v_3]=[\Univ v_1,\Univ v_3]+[\Univ v_2,\Univ v_3]\\
&[z,\Univ v]=z\\
&[\Univ v_1,\Univ v_2]+[\Univ v_2,\Univ v_1]=z\\
&[[\Univ v_1,\Univ v_2],\Univ v_3]+[[\Univ v_3,\Univ v_1],\Univ v_2]+[[\Univ v_2,\Univ v_3],\Univ v_1]=z
\end{align*}
which hold by~\cite[Theorem~3.17]{cockett:tangent-cats} and by~\cite[ Theorem~4.2]{cockett:jacobi}.
\end{proof}

\subsubsection*{The monad of vector fields}
\label{subsubsection:monad-vector-fields}
In this section we harness the additive bundle structure of Theorem~\ref{theorem:commutative-monoid-vector-fields} to define a $2$-monad structure on the $2$-endofunctor $\VF\colon\Tng(\CC)\to\Tng(\CC)$. For starters, let us assume the existence of the $n$-fold pullbacks $\VF_n(\X,\TT)$ in $\CC$. It is not hard to see that each $\VF_n$ extends to a $2$-functor. Let us define:
\begin{align*}
&\eta\colon(\X,\TT)\xrightarrow{0}\VF(\X,\TT)\\
&\mu\colon\VF^2(\X,\TT)\xrightarrow{\<\U_\VF,\VF\U\>}\VF_2(\X,\TT)\xrightarrow{+}\VF(\X,\TT)
\end{align*}

\begin{theorem}
\label{theorem:monad-structure-vector-fields}
The $2$-endofunctor $\VF\colon\Tng(\CC)\to\Tng(\CC)$ together with the $2$-natural transformations
\begin{align*}
\eta_{(\X,\TT)}\colon(\X,\TT)\to\VF(\X,\TT) &\mu_{(\X,\TT)}&\colon\VF^2(\X,\TT)\to\VF(\X,\TT)
\end{align*}
defines a $2$-monad on $\Tng(\CC)$.
\end{theorem}
\begin{proof}
Unitality and associativity of $(\VF,\eta,\mu)$ are a direct consequence of $(\U,0,+)$ being an additive bundle. We leave the reader to complete the details.
\end{proof}

In the context of tangent category theory, the $2$-monad structure of $\VF$ is the defined as follows. The unit is equal to the zero $0\colon(\X,\TT)\to\VF(\X,\TT)$ which sends each object $M$ of a tangent category to its zero vector field $(M,z_M)$. To understand the multiplication, first recall that a vector field $v\colon(M,u)\to\T^\VF(M,u)$ in the tangent category $\VF(\X,\TT)$ consists of a pair of vector fields $u,v\colon M\to\T M$ of an object $M$ of $(\X,\TT)$, which commute to each other, namely:
\begin{equation*}
\begin{tikzcd}
{\T M} & {\T^2M} \\
& {\T^2M} \\
M & {\T M}
\arrow["{\T v}", from=1-1, to=1-2]
\arrow["c"', from=2-2, to=1-2]
\arrow["u", from=3-1, to=1-1]
\arrow["v"', from=3-1, to=3-2]
\arrow["{\T u}"', from=3-2, to=2-2]
\end{tikzcd}
\end{equation*}
Therefore, the multiplication $\mu\colon\VF^2(\X,\TT)\to\VF(\X,\TT)$ sends a pair of commutative vector fields $(u,v)$ on an object $M$ of $(\X,\TT)$ to the vector field $(M,u+v)$.

\begin{remark}
\label{remark:future-work-algebras-VF}
In future work, we would like to study and classify the (lax/colax/pseudo-)algebras of the $2$-monad $\VF$.
\end{remark}


\section{Vector fields via PIE limits}
\label{section:PIE-limits}
PIE limits, namely, the $2$-limits generated by $2$-Products, Inserters, and Equifiers, are a special class of weighted limits in the context of $2$-category theory. They were introduced in~\cite{power:pie-limits} to construct the Eilenberg-Moore object $\X^S$ of a monad $(\X,S)$ internal to a $2$-category. Here, we achieve a similar result for the formal theory of tangentads: we use (some) (P)IE limits to construct vector fields.
\par In general, the $2$-category $\Tng(\CC)$ of tangentads of $\CC$ does not admit PIE limits, even when the base $2$-category $\CC$ does so, the main obstruction being the non-invertibility of the distributive law $\alpha$ of the $1$-morphisms $(F,\alpha)$ of $\Tng(\CC)$. However, we show that, when we consider strong tangent morphisms, one can lift PIE limits from the base $2$-category to the $2$-category of tangentads.

\subsection{Products}
\label{subsection:products}
For starters, let us consider $2$-products. It is not hard to believe that if a $2$-category admits $2$-products, so does the $2$-category $\Tng(\CC)$ of tangentads.

\begin{proposition}
\label{proposition:2-products}
Let $(\X,\TT)$ and $(\X',\TT')$ denote two tangentads of $\CC$. If the $2$-product $\X\times\X'$ between the underlying objects $\X$ and $\X'$ exists in $\CC$, then $\X\times\X'$ comes equipped with a tangent structure $\TT\times\TT'$ and the corresponding tangentad $(\X\times\X',\TT\times\TT')$ is the $2$-product of $(\X,\TT)$ and $(\X',\TT')$ in $\Tng(\CC)$. In particular, if $\CC$ admits all $2$-products so does $\Tng(\CC)$ and are preserved by the forgetful $2$-functor $\Tng(\CC)\to\CC$. Moreover, the projections $\pi_1\colon(\X\times\X',\TT\times\TT')\to(\X,\TT)$ and $\pi_2\colon(\X\times\X',\TT\times\TT')\to(\X',\TT')$ are strict tangent morphisms.
\end{proposition}
\begin{proof}
Let us start by showing that $\X\times\X'$ comes equipped with a tangent structure. First, notice that, from the universal property of the $2$-product $\X\times\X'$, we can construct a functor
\begin{align*}
&\End_\CC(\X)\times\End_\CC(\X')\to\End_\CC(\X\times\X')
\end{align*}
which sends a pair $(A,B)$ of endomorphisms $A\colon\X\to\X$ and $B\colon\X'\to\X'$ to the endomorphism\linebreak $A\times B\colon\X\times\X'\to\X\times\X'$ and a pair $(\varphi,\psi)$ of $2$-morphisms $\varphi\colon A\Rightarrow A'\colon\X\to\X$, $\psi\colon B\Rightarrow B'\colon\X'\to\X'$ to the $2$-morphism $\varphi\times\psi\colon A\times B\Rightarrow A'\times B'$. It is also easy to check that such a functor preserves the monoidal structure given by the composition of endomorphisms. Finally, if $c\mapsto(D_1(c),D_2(c))$ is a diagram in $\End_\CC(\X)\times\End_\CC(\X')$ and $(L_1,L_2)\to(D_1(c),D_2(c))$ is a limit cone for such a diagram, the corresponding morphism $L_1\times L_2\to D_1(c)\times D_2(c)$ is a limit cone for the diagram $c\mapsto D_1(c)\times D_2(c)$ of $\End_\CC(\X\times\X')$. Since $(\X,\TT)$ and $(\X',\TT')$ are tangentads, we can consider their Leung functors $\Leung[\TT]\colon\Weil\to\End_\CC(\X)$ and $\Leung[\TT']\colon\Weil\to\End_\CC(\X')$. Therefore, we construct a new Leung functor
\begin{align*}
\Leung[\TT\times\TT']&\colon\Weil\xrightarrow{\left\<\Leung[\TT],\Leung[\TT']\right\>}\End_\CC(\X)\times\End_\CC(\X')\to\End_\CC(\X\times\X')
\end{align*}
which defines a tangent structure $\TT\times\TT'$ for $\X\times\X'$. Furthermore, by construction, $\Leung[\TT\times\TT']$ is precisely the product of $\Leung[\TT]$ and $\Leung[\TT']$ in the category of Leung functors, whose objects are Leung functors and morphisms are monoidal natural transformations between the underlying strong monoidal functors. Thus, $(\X\times\X',\TT\times\TT')$ is the $2$-product of $(\X,\TT)$ and $(\X',\TT')$ in $\Tng(\CC)$.
\end{proof}

\subsection{Inserters}
\label{subsection:inserters}
In this section, we discuss the existence of inserters in the context of tangentads. First, let us recall the definition of an inserter in a $2$-category $\CC$. Consider a pair of $1$-morphisms $F,G\colon\X\to\X'$ and an object $\Y$ of $\CC$ and denote by $\Insert[\Y;F,G]$ the category whose objects are $1$-morphisms $H\colon\Y\to\X$ of $\CC$ together with a $2$-morphism $\tau\colon F\o H\Rightarrow G\o H$ and whose morphisms $\varphi\colon(H,\tau)\to(H',\tau')$ are $2$-morphisms $\varphi\colon H\Rightarrow H'$ of $\CC$ which commute with $\tau$ and $\tau'$, namely, making the following diagram commute:
\begin{equation*}
\begin{tikzcd}
{F\o H} & {F\o H'} \\
{G\o H} & {G\o H'}
\arrow["{F\varphi}", from=1-1, to=1-2]
\arrow["\tau"', from=1-1, to=2-1]
\arrow["{\tau'}", from=1-2, to=2-2]
\arrow["{G\varphi}"', from=2-1, to=2-2]
\end{tikzcd}
\end{equation*}
An inserter $\I\colon F\Rightarrow G$ of a pair of parallel $1$-morphisms $F,G\colon\X\to\X'$ of a $2$-category $\CC$ consists of an object $\I$ together with a $1$-morphism $V\colon\I\to\X$ and a $2$-morphism
\begin{equation*}
\begin{tikzcd}
\I & \X \\
\X & {\X'}
\arrow["V", from=1-1, to=1-2]
\arrow["V"', from=1-1, to=2-1]
\arrow["{\hat\theta}"{description}, Rightarrow, from=1-2, to=2-1]
\arrow["F", from=1-2, to=2-2]
\arrow["G"', from=2-1, to=2-2]
\end{tikzcd}
\end{equation*}
such that for each object $\Y$ of $\CC$, the induced functor
\begin{align*}
\Gamma_\I&\colon\CC[\Y,\I]\to\Insert[\Y;F,G]
\end{align*}
which sends each $1$-morphism $H\colon\Y\to\I$ to $(V\o H\colon\Y\to\X,\hat\theta_H)$ and each $2$-morphism $\psi\colon H\Rightarrow H'\colon\Y\to\I$ to $\psi_V\colon(V\o H,\hat\theta_H)\to(V\o H',\hat\theta_{H'})$ is an isomorphism of categories.
\par In general, there is no reason to believe that the $2$-category $\Tng(\CC)$ would admit all inserters even when $\CC$ does so. To understand why, consider two lax tangent morphisms $(F,\alpha),(G,\beta)\colon(\X',\TT')\to(\X,\TT)$ between two tangent categories. To make the inserter $\I$ of the two functors $F,G\colon\X\to\X'$ into a tangent category, we need to whisker the tangent bundle functor of $(\X,\TT)$ with the distributive laws as follows:
\begin{align*}
&\bar\T(M,u)=(\T M,F\T M\xrightarrow{\alpha}\T'FM\xrightarrow{\T'u}\T'GM\xleftarrow{\beta}G\T M)
\end{align*}
However, the distributive lax $\beta$ is pointing in the wrong direction. The solution is to assume $(G,\beta)$ to be strong. Furthermore, we also need to assume that the functor $G$ preserves the $n$-fold pullback of the projection along itself and the universal property of the vertical lift.

\begin{definition}
\label{definition:tangent-preserving-morphism}
A $1$-morphism $G\colon\X\to\X'$ \textbf{preserves the fundamental tangent limits} of a tangentad $(\X,\TT)$ provided that each functor $\CC[\X\|G,\beta]\colon\End_\CC(\X)\to\CC[\X\|\X']$ preserves the pointwise $n$-fold pullback of the projection $p\colon\T\Rightarrow\id_\X$ along itself and the pointwise pullback of the universality of the vertical lift $l\colon\T\Rightarrow\T^2$ of $(\X,\TT)$. We also say that a tangent morphism $(G,\beta)\colon(\X,\TT)\to(\X',\TT')$ preserves the fundamental tangent limits provided that its underlying $1$-morphism $G$ preserves the fundamental tangent limits of $(\X,\TT)$.
\end{definition}

We start with a lemma.

\begin{lemma}
\label{lemma:inserters}
Consider two parallel lax tangent morphisms $(F,\alpha),(G,\beta)\colon(\X,\TT)\to(\X',\TT')$ of a $2$-category $\CC$ and let $\Y$ be an object of $\CC$. If $(G,\beta)$ is a strong tangent morphism that preserves the fundamental tangent limits, the category $\Insert[\Y;F,G]$ is a tangent category where the tangent bundle functor sends an object $(H\colon\Y\to\X,\tau)$ to $(\T\o H,\beta^{-1}\o\tau_\T\o\alpha)$.
\end{lemma}
\begin{proof}
Let $\Y$ be an object of $\CC$. The category $\Insert[\Y;F,G]$ is a subcategory of the Hom-category $\CC[\Y,\X]$, which is a tangent category with the pointwise tangent structure induced by $(\X,\TT)$. Consider an object $(H,\tau)$ of $\Insert[\Y;F,G]$. It is clear that $(\T\o H,\beta^{-1}\o\tau_\T\o\alpha)$ is also an object of $\Insert[\Y;F,G]$. Moreover, by the compatibilities between the distributive laws $\alpha$ and $\beta$ with the tangent structures, it is easy to see that the structural natural transformations of the Hom-tangent category $[\Y\|\X,\TT]$ restrict to $\Insert[\Y;F,G]$. Moreover, since the fundamental tangent limits of $(\X,\TT)$, namely, the $n$-fold pullback of the projection along itself and the universality of the vertical lift, are pointwise limits, they are preserved by pre-composition with any $H\colon\Y\to\X$, thus, $(\T_n\o H,\beta^{-1}_n\o\tau_{\T_n}\o\alpha_n)$ is the $n$-fold pullback of the projection in $\Insert[\Y;F,G]$, where:
\begin{align*}
\alpha_n\=\<\alpha\o F\pi_1\,\alpha\o F\pi_n\>&\colon F\o\T_n\Rightarrow\T'_n\o F\\
\beta_n^{-1}\=\<\beta^{-1}\o{\pi_1}_G,\beta^{-1}\o{\pi_n}_G\>&\colon\T'_n\o G\Rightarrow G\o\T_n
\end{align*}
Similarly, the vertical lift of $\Insert[\Y;F,G]$ is universal. Thus, $\Insert[\Y;F,G]$ is a tangent subcategory of $[\Y\|\X,\TT]$.
\end{proof}

Thanks to Lemma~\ref{lemma:inserters}, we can prove that inserters from lax tangent morphisms to strong tangent morphisms which preserve the fundamental tangent limits exist in $\Tng(\CC)$ provided they exist in $\CC$.

\begin{proposition}
\label{proposition:inserters}
Consider two parallel lax tangent morphisms $(F,\alpha),(G,\beta)\colon(\X,\TT)\to(\X',\TT')$ of a $2$-category $\CC$. Assume also that $(G,\beta)$ is a strong tangent morphism which preserves the fundamental tangent limits. If there exists an inserter $\I\colon F\Rightarrow G$ of $F$ and $G$ in $\CC$, then $\I$ comes equipped with a tangent structure $\TT^\I$ such that the universal $1$-morphism $V\colon\I\to\X$ becomes a strict tangent morphism that preserves the fundamental tangent limits, the universal $2$-morphism $\hat\theta\colon F\o V\Rightarrow G\o G$ becomes a tangent $2$-morphism and $(\I,\TT^\I)\colon(F,\alpha)\Rightarrow(G,\beta)$ becomes an inserter of $(F,\alpha)$ and $(G,\beta)$ in $\Tng(\CC)$. In particular, if $\CC$ admits inserters, so does $\Tng_\cong(\CC)$.
\end{proposition}
\begin{proof}
Thanks to Lemma~\ref{lemma:inserters}, for each object $\Y$ of $\CC$, the category $\Insert[\Y;F,G]$ comes equipped with a tangent structure. This construction extends to a $2$-functor
\begin{align*}
&\Insert[-;F,G]\colon\CC^\op\to\TngCat
\end{align*}
However, by the universal property of the inserter, $\Insert[-;F,G]\cong\CC[-,\I]$ is corepresentable. Therefore, by Corollary~\ref{corollary:tangentad-representable-2-functors}, it is equivalent to a tangent structure $\TT^\I$ on $\I$. Now, consider a lax tangent morphism $(H,\gamma)\colon(\X'',\TT'')\to(\X,\TT)$ together with a tangent $2$-morphism $\tau\colon(F,\alpha)\o(H,\gamma)\Rightarrow(G,\beta)\o(H,\gamma)$. By the universal property of the inserter, there exists a $1$-morphism $\bar H\colon\X''\to\I$ such that:
\begin{align*}
&V\o\bar H=H\\
&\hat\theta_H=\tau
\end{align*}
The $2$-morphism
\begin{align*}
&V\o\bar H\o\T''=H\o\T''\xrightarrow{\gamma}\T\o H=\T\o V\o\bar H
\end{align*}
becomes a morphism of $\Insert[\X''\|F,G]$, thus, by the universal property of the inserter, it induces a $2$-morphism
\begin{align*}
\bar\gamma&\colon\bar H\o\T''\Rightarrow\T\o\bar H
\end{align*}
which makes $(\bar H,\bar\gamma)\colon(\X'',\TT'')\to(\X,\TT)$ into a lax tangent morphism. Similarly, one can use the universal property of the inserter to prove that a tangent $2$-morphism $\varphi\colon(H,\gamma)\Rightarrow(H',\gamma')$ which commutes with $\tau$ and $\tau'$ induces a unique tangent $2$-morphism $\bar\varphi\colon(\bar H,\bar\gamma)\Rightarrow(\bar H',\bar\gamma')$.
\end{proof}

\subsection{Equifiers}
\label{subsection:equifiers}
In the previous section, we proved the existence of inserters in the $2$-category $\Tng(\CC)$ from lax tangent morphisms to strong tangent morphisms which preserve the fundamental tangent limits, provided the existence of these inserters in $\CC$. In this section, we discuss the existence of equifiers in $\Tng(\CC)$. First, recall that an equifier $\E\colon\varphi\to\psi$ between two $2$-morphisms $\varphi,\psi\colon F\Rightarrow G\colon\X\to\X'$ in a $2$-category $\CC$ consists of an object $\E$ together with a $1$-morphism $W\colon\E\to\X$ such that $\varphi_W=\psi_W$ and the induced functor
\begin{align*}
&\CC(\Y,W)\colon\CC(\Y,\E)\to\Equif(\Y;\varphi,\psi)
\end{align*}
is an isomorphism of categories, where $\Equif(\Y;\varphi,\psi)$ is the category whose objects are $1$-morphisms $H\colon\Y\to\X$ such that $\varphi_H=\psi_H$ and morphisms are $2$-morphisms $\xi\colon H\Rightarrow H'$.
\par In the context of categories, the equifier between two natural transformations $\varphi,\psi\colon F\Rightarrow G\colon\X\to\X'$ is the subcategory $\E$ of $\X$ of objects $M$ for which $\varphi_M=\psi_M$ and all morphisms of $\X$.
\par As for inserters, the tangent structures do not lift to the equifier of two tangent $2$-morphisms between lax tangent morphisms. However, this is the case when the target tangent morphism is strong and preserves the fundamental tangent limits.

\begin{lemma}
\label{lemma:equifiers}
Consider two parallel lax tangent morphisms $(F,\alpha),(G,\beta)\colon(\X,\TT)\to(\X',\TT')$ of a $2$-category $\CC$, two tangent $2$-morphisms $\varphi,\psi\colon(F,\alpha)\Rightarrow(G,\beta)$, and let $\Y$ be an object of $\CC$. If $(G,\beta)$ is a strong tangent morphism that preserves the fundamental tangent limits, the category $\Equif[\Y;\varphi,\psi]$ is a tangent category where the tangent bundle functor sends an object $H\colon\Y\to\X$ to $\T\o H$.
\end{lemma}
\begin{proof}
As per $\Insert[\Y;F,G]$, the category $\Equif[\Y;\varphi,\psi]$ is a subcategory of the Hom-tangent category $[\Y\|\X,\TT]$. Consider an object of $\Equif[\Y;\varphi,\psi]$, namely, a $1$-morphism $H\colon\Y\to\X$ such that $\varphi_H=\psi_H$. We want to prove that $\varphi_{\T\o H}=\psi_{\T\o H}$
\begin{align*}
&\quad\varphi_{\T\o H}\\
=&\quad(\varphi_\T)_H            &&(\varphi_\T=\beta^{-1}\o\T'\varphi\o\alpha)\\
=&\quad(\beta^{-1}\o\T'\varphi\o\alpha)_H\\
=&\quad\beta^{-1}_H\o\T'\varphi_H\o\alpha_H       &&(\varphi_H=\psi_H)\\
=&\quad\beta^{-1}_H\o\T'\psi_H\o\alpha_H\\
=&\quad(\beta^{-1}\o\T'\psi\o\alpha)_H       &&(\beta^{-1}\o\T'\psi\o\alpha=\psi_\T)\\
=&\quad(\psi_\T)_H\\
=&\quad\psi_{\T\o H}
\end{align*}
where we used that $\varphi$ and $\psi$ are tangent $2$-morphisms and that $\beta$ is invertible. Similarly, using that $G$ preserves the fundamental tangent limits, it is not hard to see that $-\o\T_n$ defines the $n$-fold pullback of the projection along itself and that the vertical lift is universal. Thus, $\Equif[\Y;\varphi,\psi]$ is a tangent subcategory of the Hom-tangent category $[\Y\|\X,\TT]$.
\end{proof}

\begin{proposition}
\label{proposition:equifiers}
Consider two parallel lax tangent morphisms $(F,\alpha),(G,\beta)\colon(\X,\TT)\to(\X',\TT')$ of a $2$-category $\CC$ and two tangent $2$-morphisms $\varphi,\psi\colon(F,\alpha)\Rightarrow(G,\beta)$. Assume also that $(G,\beta)$ is a strong tangent morphism which preserves the fundamental tangent limits. If there exists an equifier $\E\colon\varphi\to\psi$ of $\varphi$ and $\psi$ in $\CC$, then $\E$ comes equipped with a tangent structure $\TT^\E$ such that the universal $1$-morphism $W\colon\E\to\X$ becomes a strict tangent morphism that preserves the fundamental tangent limits and $(\E,\TT^\E)\colon\varphi\to\psi$ becomes an equifier of $\varphi$ and $\psi$ in $\Tng(\CC)$. In particular, if $\CC$ admits equifiers, so does $\Tng_\cong(\CC)$.
\end{proposition}
\begin{proof}
Thanks to Lemma~\ref{lemma:equifiers}, for each object $\Y$ of $\CC$, the category $\Equif[\Y;\varphi,\psi]$ comes equipped with a tangent structure. This construction extends to a $2$-functor
\begin{align*}
&\Equif[-;\varphi,\psi]\colon\CC^\op\to\TngCat
\end{align*}
However, by the universal property of the equifier, $\Equif[-;\varphi,\psi]\cong\CC[-,\E]$ is corepresentable. Therefore, by Corollary~\ref{corollary:tangentad-representable-2-functors}, it is equivalent to a tangent structure $\TT^\E$ on $\E$. Now, consider a lax tangent morphism $(H,\gamma)\colon(\X'',\TT'')\to(\X,\TT)$ such that $\varphi_H=\psi_H$. By the universal property of the equifier, there exists a unique $\bar H\colon\X''\to\E$ such that $W\o\bar H=H$. Moreover, $\gamma\colon H\o\T''\Rightarrow\T\o H$ becomes a morphism of $\Equif[\X'';\varphi,\psi]$. Thus, by the universal property of the equifier, it corresponds to a $2$-morphism $\bar\gamma\colon\bar H\o\T''\to\T^\E\o\bar H$, which makes $(\bar\H,\bar\gamma)$ into a lax tangent morphism. Finally, the universal property of the equifier implies that every tangent $2$-morphism $\xi\colon(H,\gamma)\Rightarrow(H',\gamma')$ corresponds to a tangent $2$-morphism $\bar\xi\colon(\bar H,\bar\gamma)\Rightarrow(\bar H',\bar\gamma')$.
\end{proof}

\subsection{Pullbacks}
\label{subsection:pullbacks}
In various constructions, we require the existence of some $2$-pullbacks in the $2$-category $\Tng(\CC)$. In this section, we study which $2$-pullbacks in the base $2$-category $\CC$ lift to $\Tng(\CC)$.
\par Consider two strict tangent morphisms $F\colon(\X_1,\TT_1)\to(\X_0,\TT_0)$ and $G\colon(\X_2,\TT_2)\to(\X_0,\TT_0)$ and let us assume the $2$-pullback
\begin{equation*}
\begin{tikzcd}
\X & {\X_2} \\
{\X_1} & {\X_0}
\arrow["{\Pi_2}", from=1-1, to=1-2]
\arrow["{\Pi_1}"', from=1-1, to=2-1]
\arrow["\lrcorner"{anchor=center, pos=0.125}, draw=none, from=1-1, to=2-2]
\arrow["G", from=1-2, to=2-2]
\arrow["F"', from=2-1, to=2-2]
\end{tikzcd}
\end{equation*}
exists in $\CC$. Consider also the following $2$-pullback in $\Cat$:
\begin{equation}
\label{equation:base-pullback}
\begin{tikzcd}
{\End(\X_1)\times_{\CC[\X,\X_0]}\End(\X_2)} && {\End(\X_2)} \\
&& {\CC[\X,\X_2]} \\
{\End(\X_1)} & {\CC[\X,\X_1]} & {\CC[\X,\X_0]}
\arrow[from=1-1, to=1-3]
\arrow[from=1-1, to=3-1]
\arrow["\lrcorner"{anchor=center, pos=0.125}, draw=none, from=1-1, to=3-2]
\arrow["{\CC[\Pi_2,\X_2]}", from=1-3, to=2-3]
\arrow["{\CC[\X,G]}", from=2-3, to=3-3]
\arrow["{\CC[\Pi_1,\X_1]}"', from=3-1, to=3-2]
\arrow["{\CC[\X,F]}"', from=3-2, to=3-3]
\end{tikzcd}
\end{equation}
The category $\End(\X_1)\times_{\CC[\X,\X_0]}\End(\X_2)$ has per objects pairs $(A,B)$ formed by endomorphisms $A\colon\X_1\to\X_1$ and $B\colon\X_2\to\X_2$ such that the following diagram commutes:
\begin{equation*}
\begin{tikzcd}
\X && {\X_2} \\
&& {\X_2} \\
{\X_1} & {\X_1} & {\X_0}
\arrow["{\Pi_2}", from=1-1, to=1-3]
\arrow["{\Pi_1}"', from=1-1, to=3-1]
\arrow["B", from=1-3, to=2-3]
\arrow["G", from=2-3, to=3-3]
\arrow["A"', from=3-1, to=3-2]
\arrow["F"', from=3-2, to=3-3]
\end{tikzcd}
\end{equation*}
and morphisms are pairs $(\varphi,\psi)\colon(A,B)\to(A',B')$ of $2$-morphisms $\varphi\colon A\Rightarrow A'$ and $\psi\colon B\Rightarrow B'$ such that:
\begin{align*}
&F\varphi_{\Pi_1}=G\psi_{\Pi_2}\colon F\o A\o\Pi_1=G\o B\o\Pi_2\to G\o B'\o\Pi_2=F\o A'\o\Pi_1
\end{align*}
Using the universal property of the $2$-pullback of Equation~\eqref{equation:base-pullback}, we can define a functor
\begin{align*}
&\End(\X_1)\times_{\CC[\X,\X_0]}\End(\X_2)\to\End(\X)
\end{align*}
which sends each pair $(A,B)$ to $A\times_{\X_0}B\colon\X\to\X$ and each morphism $(\varphi,\psi)\colon(A,B)\to(A',B')$ to $\varphi\times_{\X_0}\psi\colon A\times_{\X_0}B\to A'\times_{\X_0}B'$.

\begin{proposition}
\label{proposition:pullbacks}
If $F\colon(\X_1,\TT_1)\to(\X_0,\TT_0)$ and $G\colon(\X_2,\TT_2)\to(\X_0,\TT_0)$ are two strict tangent morphisms which preserve the fundamental tangent limits and the $2$-pullback of Equation~\eqref{equation:base-pullback} exists in $\CC$, then the $2$-pullback $\X$ of $F$ along $G$ comes equipped with a tangent structure $\TT$ which makes the projections $\Pi_k\colon(\X,\TT)\to(\X_k,\TT_k)$ ($k=1,2$) strict tangent morphisms which preserve the fundamental tangent limits and $(\X,\TT)$ becomes the $2$-pullback of $F$ and $G$ in $\Tng(\CC)$.
\end{proposition}
\begin{proof}
Since $F$ and $G$ are strict tangent morphisms, $(\T_1,\T_2)$ becomes an object of $\End(\X_1)\times_{\CC[\X,\X_0]}\End(\X_2)$ since:
\begin{align*}
&F\o\T_1\o\Pi_1=\T_0\o F\o\Pi_1=\T_0\o G\o\Pi_2=G\o\T_2\o\Pi_2
\end{align*}
Moreover, since $F$ and $G$ preserves the fundamental tangent limits, so do each $((\T_1)_n,(\T_2)_n)$. Using a similar argument, $(p_1,p_2)$, $(z_1,z_2)$, $(s_1,s_2)$, $(l_1,l_2)$, and $(c_1,c_2)$ becomes morphisms of $\End(\X_1)\times_{\CC[\X,\X_0]}\End(\X_2)$. Therefore, we can define a functor:
\begin{align*}
&\Weil\xrightarrow{\<\Leung[\TT_1],\Leung[\TT_2]\>}\End(\X_1)\times_{\CC[\X,\X_0]}\End(\X_2)
\end{align*}
Moreover, we can postcompose such a functor by the functor $\End(\X_1)\times_{\CC[\X,\X_0]}\End(\X_2)\to\End(\X)$ induced by the universal property of $\X$ and obtain a functor:
\begin{align*}
\Leung[\TT]&\colon\Weil\xrightarrow{\<\Leung[\TT_1],\Leung[\TT_2]>}\End(\X_1)\times_{\CC[\X,\X_0]}\End(\X_2)\to\End(\X)
\end{align*}
It is not hard to see that this is a strong monoidal functor. To prove that this functor preserves the pointwise fundamental tangent limits, consider an endomorphism $F\colon\X\to\X$ and the following commutative diagram of $2$-morphisms:
\begin{equation*}
\begin{tikzcd}
F \\
& {\T_n} & \T \\
& \T & {\id_\X}
\arrow["{\varphi_n}", curve={height=-12pt}, from=1-1, to=2-3]
\arrow["{\varphi_1}"', curve={height=12pt}, from=1-1, to=3-2]
\arrow["{\pi_n}", from=2-2, to=2-3]
\arrow["{\pi_1}"', from=2-2, to=3-2]
\arrow["p", from=2-3, to=3-3]
\arrow["p"', from=3-2, to=3-3]
\end{tikzcd}
\end{equation*}
where $p$ and $\pi_k$ are $\Leung[\TT](p)$ and $\Leung[\TT](\pi_k)$, respectively. By postcomposing the diagram by each $\Pi_j$, ($j=1,2$) we obtain the diagram:
\begin{equation*}
\begin{tikzcd}
{\Pi_j\o F} \\
& {\Pi_j\o\T_n} &&& {\Pi_j\o \T} \\
&& {(\T_j)_n\o\Pi_j} & {\T_j\o\Pi_j} \\
&& {\T_j\o\Pi_j} & {\Pi_j} \\
& {\Pi_j\o \T} &&& {\Pi_j}
\arrow["{\Pi_j\varphi_n}", curve={height=-12pt}, from=1-1, to=2-5]
\arrow["{\Pi_j\varphi_1}"', curve={height=12pt}, from=1-1, to=5-2]
\arrow["{\Pi_j\pi_n}", from=2-2, to=2-5]
\arrow[equals, from=2-2, to=3-3]
\arrow["{\Pi_j\pi_1}"', from=2-2, to=5-2]
\arrow[equals, from=2-5, to=3-4]
\arrow["{\Pi_jp}", from=2-5, to=5-5]
\arrow["{p_n\Pi_j}", from=3-3, to=3-4]
\arrow["{\pi_1\Pi_j}"', from=3-3, to=4-3]
\arrow["\lrcorner"{anchor=center, pos=0.125}, draw=none, from=3-3, to=4-4]
\arrow["{p_j\Pi_j}", from=3-4, to=4-4]
\arrow["{p_j\Pi_j}"', from=4-3, to=4-4]
\arrow[equals, from=4-4, to=5-5]
\arrow[equals, from=5-2, to=4-3]
\arrow["{\Pi_jp}"', from=5-2, to=5-5]
\end{tikzcd}
\end{equation*}
However, the internal square is a pullback diagram, thus, there exists a unique $2$-morphism $\xi_j\colon\Pi_j\o F\Rightarrow(\T_j)_n\o\Pi_j$, for each $j=1,2$. Thus, by the universal property of the $2$-pullback of Equation~\eqref{equation:base-pullback}, there exists a unique morphism $\<\varphi_1,\varphi_n\>\colon F\Rightarrow\T_n$ such that:
\begin{align*}
&\Pi_j\<\varphi_1,\varphi_n\>=\xi_j
\end{align*}
In particular, is the unique $2$-morphism $\<\varphi_1,\varphi_n\>$ satisfying the equations
\begin{align*}
&\pi_k\o\<\varphi_1,\varphi_n\>=\varphi_k
\end{align*}
for $k=1\,n$. Thus, $\T_n$ is the $n$-fold pullback of $p$ along itself. To prove that such a pullback is pointwise, one needs to precompose the previous diagram by an arbitrary morphism $G\colon\Y\to\X$ and use that the $n$-fold pullbacks of $p_j\colon\T_j\to\id_{\X_j}$ are pointwise. Finally, using a similar strategy, one also shows that the vertical lift is universal. Thus, $(\X,\TT)$ is a tangent category. Finally, to prove that $(\X,\TT)$ is a $2$-pullback in $\Tng(\CC)$, one uses the universal property of the $2$-pullback~\eqref{equation:base-pullback} and that the projections $\Pi_j$ are strict tangent morphisms, which is a consequence of the construction of $(\X,\TT)$.
\end{proof}

\subsection{The construction of vector fields via (P)IE limits}
\label{subsection:PIE-limits-vector-fields}
So far, we discussed under what conditions the PIE limits of a $2$-category lift to the $2$-category $\Tng(\CC)$ of tangentads of $\CC$. In this section, we employ inserters and equifiers to construct vector fields for tangentads. For starters, consider a tangent category $(\X,\TT)$. Consider also the following two parallel tangent morphisms:
\begin{align*}
&\id_{(\X,\TT)},(\T,c)\colon(\X,\TT)\to(\X,\TT)
\end{align*}
Recall that the $2$-category $\Cat$ of categories admits inserters and equifiers, and notice that $(\T,c)$ is a strong tangent morphism which preserves the fundamental tangent limits. Thus, by Proposition~\ref{proposition:inserters}, the inserter $(\I,\TT^\I)\colon\id_{(\X,\TT)}\Rightarrow(\T,c)$ of $\id_{(\X,\TT)}$ and $(\T,c)$ is well-defined in $\TngCat$. Concretely, $\I$ is the category of pairs $(M,v)$, formed by an object $M$ of $(\X,\TT)$ and a morphism $v\colon M=\id_\X(M)\to\T M$. A morphism $f\colon(M,v)\to(N,u)$ is a morphism $f\colon M\to N$ of $(\X,\TT)$ which commutes with $v$ and $u$, namely, $u\o f=\T f\o v$. By unwrapping the proof of Proposition~\ref{proposition:inserters}, one can also show that the tangent bundle functor $\T^\I$ sends each $(M,v)$ to $(\T M,v_\T)$, where
\begin{align*}
&v_\T\colon\T M\xrightarrow{\T v}\T^2M\xrightarrow{c}\T^2M
\end{align*}
and each morphism $f$ to $\T f$. Finally, the structural natural transformations of $\TT^\I$ are defined as in $(\X,\TT)$. Notice that $(\I,\TT^\I)$ resembles the tangent category of vector fields of $(\X,\TT)$ with only one difference: the morphisms $v\colon M\to\T M$ need not satisfy any constraint. To impose each $v$ to be a section of the projection, namely, to satisfy the equation $p\o v=\id_M$, we need to use an equifier. Consider the two parallel tangent morphisms
\begin{align*}
&V,(\T,c)\o V\colon(\I,\TT^\I)\to(\X,\TT)
\end{align*}
where $V\colon(\I,\TT^\I)\to(\X,\TT)$ is the universal strict tangent morphism of the inserter. Consider also the two tangent $2$-morphisms
\begin{align*}
&V\xrightarrow{\hat\theta}V\o(\T,c)\xrightarrow{Vp}V&&\id_V\colon V\to V
\end{align*}
where $\hat\theta\colon V\o\id_{(\X,\TT)}\Rightarrow V\o(\T, c)$ is the universal $2$-morphism of the inserter $(\I,\TT^\I)$. Let $(\E,\TT^\E)$ be the equifier of $Vp\o\hat\theta$ and $\id_V$ in $\Tng(\CC)$. Concretely, $(\E,\TT^\E)$ is the tangent subcategory of $(\I,\TT^\I)$ spanned by those objects $(M,v)$ of $(\I,\TT^\I)$ satisfying
\begin{align*}
(Vp\o\hat\theta)_{(M,v)}=\id_{(M,v)}
\end{align*}
namely:
\begin{align*}
&(M\xrightarrow{v}\T M\xrightarrow{p_M}M)=(M\xrightarrow{\id_M}M)
\end{align*}
Thus, $(\E,\TT^\E)$ is exactly the tangent category of vector fields of $(\X,\TT)$.

\begin{theorem}
\label{theorem:PIE-limits-vector-fields}
Consider a tangentad $(\X,\TT)$ in a $2$-category $\CC$ and suppose that the inserter $\I\colon\id_\X\Rightarrow\T$ and the equifier $\E\colon Vp\o\hat\theta\Rightarrow\id_V$ exist in $\CC$. Thus, $(\X,\TT)$ admits the construction of vector fields in $\CC$ and the tangentad $\VF(\X,\TT)$ of vector fields of $(\X,\TT)$ is the equifier $(\E,\TT^\E)$.
\end{theorem}
\begin{proof}
Thanks to Propositions~\ref{proposition:inserters} and~\ref{proposition:equifiers}, $(\I,\TT^\I)$ and $(\E,\TT^\E)$ are an inserter and an equifier in $\Tng(\CC)$, respectively. Consider a tangentad $(\X',\TT')$ and let us unwrap the definition of the category $\Equif[\X',\TT';V p\o\hat\theta,\id_V]$. An object of this category is a lax tangent morphism $(F,\alpha)\colon(\X',\TT')\to(\I,\TT^\I)$ satisfying the following equation:
\begin{align*}
&(V p\o\hat\theta)_F=\id_{V\o F}
\end{align*}
A morphism of $\Equif[\X',\TT';V p\o\hat\theta,\id_V]$ is a tangent $2$-morphism $\varphi\colon(F,\alpha)\Rightarrow(G,\beta)$. By postcomposing an object $(F,\alpha)$ by $V\colon(\I,\TT^\I)\to(\X,\TT)$ we obtain a lax tangent morphism $V\o(F,\alpha)\colon(\X',\TT')\to(\X,\TT)$ equipped with a $2$-morphism:
\begin{align*}
&v\colon V\o(F,\alpha)\xrightarrow{\hat\theta_{(F,\alpha)}}(\T,c)\o V\o(F,\alpha)=\bar\T(V\o(F,\alpha))
\end{align*}
Moreover, since $(V p\o\hat\theta)_F=\id_{V\o F}$, the $2$-morphism $v\colon V\o(F,\alpha)\to\bar\T(V\o(F,\alpha))$ defines a vector field in the Hom-tangent category $[\X',\TT'\|\X,\TT]$. This construction extends to a functor
\begin{align*}
&\Equif[\X',\TT';V p\o\hat\theta,\id_V]\to\VF[\X',\TT'\|\X,\TT]
\end{align*}
Conversely, given a vector field $v\colon(H,\gamma)\to\bar\T(H,\gamma)$ in the Hom-tangent category $[\X',\TT'\|\X,\TT]$, by virtue of the universal property of the inserter, there exists a tangent morphism $(\bar H,\bar\gamma)\colon(\X',\TT')\to(\I,\TT^\I)$ such that $V\o(\bar H,\bar\gamma)=(H,\gamma)$ and $\hat\theta_{\bar H}=v$. Furthermore, since $p_H\o v=\id_{(H,\gamma)}$:
\begin{align*}
&(Vp\o\hat\theta)_{\bar H}=Vp_{\bar H}\o\hat\theta_{\bar H}=p_{V\o\bar H}\o v=p_H\o v=\id_H=(\id_V)_{\bar H}
\end{align*}
Therefore, by the universal property of the inserter, $(\bar H,\bar\gamma)$ is an object of $\Equif[\X',\TT';Vp\o\hat\theta,\id_V]$. One can show that this defines an isomorphism:
\begin{align*}
&\Equif[\X',\TT';Vp\o\hat\theta,\id_V]\cong\VF[\X',\TT'\|\X,\TT]
\end{align*}
However, by the universal property of the equifier, $\Equif[\X',\TT';Vp\o\hat\theta,\id_V]$ is also isomorphic to the Hom-tangent category $[\X',\TT'\|\E,\TT^\E]$. Thus, the equifier $(\E,\TT^\E)$ is the tangentad of vector fields of $(\X,\TT)$.
\end{proof}

When a $2$-category $\CC$ admits inserters and equifiers, by Proposition~\ref{theorem:PIE-limits-vector-fields}, we can construct vector fields for any tangentads.

\begin{corollary}
\label{corollary:PIE-limits-vector-fields}
A $2$-category $\CC$ which admits inserters and equifiers admits the construction of vector fields.
\end{corollary}


\section{Applications}
\label{subsection:examples-vector-fields}
In Section~\ref{subsection:examples-tangentads}, we listed some examples of tangentads. In this section, we compute the construction of vector fields for each of these examples.

\subsection{Tangent monads}
\label{subsubsection:vector-fields-tangent-monads}
Consider a tangent monad $(S,\alpha)$ on a tangent category $(\X,\TT)$. Since $(S,\alpha)$ is a lax tangent morphism, by Proposition~\ref{proposition:VF-functoriality}, $(S,\alpha)$ sends each vector field $(M,v)$ of $(\X,\TT)$ to a vector field $(SM,v_S)$, where $v_S$ is defined by:
\begin{align*}
&SM\xrightarrow{Sv}S\T M\xrightarrow{\alpha}\T SM
\end{align*}
Furthermore, if $f\colon(M,v)\to(M',v')$ is a morphism of vector fields, by naturality of $\alpha$, so is $Sf\colon(SM,v_S)\to(SM',v'_S)$. Since the unit $\eta$ and the multiplication $\mu$ of the monad $S$ are compatible with $\alpha$, the monad $S$ lifts to a monad $\VF(S,\alpha)$ on $\VF(\X,\TT)$ which is compatible with the tangent structure. Thus, $(\VF(\X,\TT),\VF(S,\alpha))$ defines a tangent monad.

\begin{theorem}
\label{theorem:vector-fields-tangent-monads}
The $2$-category $\Mnd$ of monads admits the construction of vector fields. In particular, the tangentad of vector fields of a tangent monad $(S,\alpha)$ on a tangent category $(\X,\TT)$ is the tangent monad $\VF(S,\alpha)$.
\end{theorem}
\begin{proof}
Even though the $2$-category $\Mnd$ does not admit all inserters and equifiers, it admits inserters $\I\colon F\Rightarrow G$ and equifiers $\E\colon\varphi\to\psi\colon F\Rightarrow G$ whose base $1$-morphism $F$ is a strong morphism of monads, namely, whose distributive law with the monad is invertible (in~\cite[Propositions~4.3 and~4.3]{lack:limits-lax-morphisms} it was proven when $F$ is strict, but the proof extends readily for strong morphisms). Luckily, the inserters and equifiers required to construct the tangentads of vector fields fit into this class. Thus, by Theorem~\ref{theorem:PIE-limits-vector-fields}, $\Mnd$ admits the construction of vector fields. By unwrapping the definitions of inserters and equifiers in $\Mnd$, one finds out that the tangentad of vector fields of a tangent monad $(S,\alpha)$ is precisely $\VF(S,\alpha)$.
\end{proof}

\subsection{Tangent fibrations}
\label{subsubsection:vector-fields-tangent-fibrations}
In this section, we compute the construction of vector fields for tangent fibrations. To begin, let us prove that the $2$-category $\Fib$ of fibrations admits PIE limits.

\begin{lemma}
\label{lemma:PIE-limits-fibrations}
The $2$-category $\Fib$ of cloven fibrations over arbitrary base categories admits PIE limits.
\end{lemma}
\begin{proof}
First, recall that the $2$-category $\Fib(\X)$ of cloven fibrations over a fixed base is pseudo-monadic~\cite{street:fibrations-yoneda-lemma} and that the $2$-category of pseudo-algebras of a $2$-monad lifts PIE limits~\cite[Section~2]{blackwell:2-monad-theory}. Thus, $\Fib(\X)$ admits inserters and equifiers. As suggested by Richard Garner in an informal conversation with the author, one can extend the $2$-monad that classifies cloven fibrations over a fixed base to a $2$-monad on the entire $2$-category and prove that also $\Fib$ is pseudo-monadic. Another approach is to calculate the PIE limits directly. This is a tedious but not difficult task that we leave the reader to complete.
\end{proof}

\begin{remark}
\label{remark:PIE-limits-fibrations}
We thank Quentin Schroeder for pointing out that Street's theorem~\cite{street:fibrations-yoneda-lemma} only proves pseudo-monadicity for fibrations on a fixed base. We also thank Richard Garner and Luca Mesiti for suggesting solutions to this problem.
\end{remark}

\begin{theorem}
\label{theorem:vector-fields-tangent-fibrations}
The $2$-category $\Fib$ of fibrations admits the construction of vector fields. In particular, the tangentad of vector fields of a tangent fibration $\Pi\colon(\X',\TT')\to(\X,\TT)$ is the tangent fibration $\VF(\Pi)\colon\VF(\X',\TT')\to\VF(\X,\TT)$ which sends a vector field $(E,w)$ of $(\X',\TT')$ to $(\Pi(E),\Pi(w))$.
\end{theorem}
\begin{proof}
 Thanks to Lemma~\ref{lemma:PIE-limits-fibrations}, $\Fib$ admits PIE limits, thus, by Theorem~\ref{theorem:PIE-limits-vector-fields}, $\Fib$ admits the construction of vector fields. By unwrapping this construction, one finds out that the tangentad of vector fields of a tangent fibration $\Pi$ is precisely $\VF(\Pi)$.
\end{proof}

\subsection{Tangent indexed categories}
\label{subsubsection:vector-fields-tangent-categories}
\cite[Theorem~5.5]{lanfranchi:grothendieck-tangent-cats} proves that the $2$-categories $\TngFib$ and $\TngIndx$ of tangent fibrations and tangent indexed categories are equivalent via a Grothendieck-like construction. In this section, we utilize this equivalence to construct the vector fields of tangent indexed categories. First, let us prove that equivalences of tangentads induce equivalences of the tangentads of vector fields.

\begin{proposition}
\label{proposition:equivalence-vector-fields}
Let $\CC$ and $\CC'$ be two $2$-categories and suppose there is a $2$-equivalence $\Xi\colon\Tng(\CC)\simeq\Tng(\CC')$ of the $2$-categories of tangentads of $\CC$ and $\CC'$. If a tangentad $(\X,\TT)$ of $\CC$ admits the construction of vector fields and $\VF(\X,\TT)$ denotes the tangentad of vector fields of $(\X,\TT)$, so does $\Xi(\X,\TT)$ and the tangentad of vector fields of $\Xi(\X,\TT)$ is $\Xi(\VF(\X,\TT))$.
\end{proposition}
\begin{proof}
First, notice that a $2$-equivalence $\Xi\colon\Tng(\CC)\simeq\Tng(\CC')$ lifts to a $2$-equivalence
\begin{align*}
&\VF(\Xi)\colon\VF[\X',\TT'\|\X,\TT]\to\VF[\Xi(\X',\TT')\|\Xi(\X,\TT)]
\end{align*}
between the tangent category of vector fields in each Hom-tangent category $[\X',\TT'\|\X,\TT]$ of $\Tng(\CC)$ to the one of vector fields in the corresponding Hom-tangent category $[\Xi(\X',\TT')\|\Xi(\X,\TT)]$ of $\Tng(\CC')$. Furthermore, the functor
\begin{align*}
&\Gamma_{\VF(\Xi)(\Univ v)}\colon[\Xi(\X',\TT')\|\Xi(\VF(\X,\TT))]\to\VF[\Xi(\X',\TT')\|\Xi(\X,\TT)]
\end{align*}
coincides with the functor $\Xi(\Gamma_{\Univ v})$, thus, in particular, $\Gamma_{\VF(\Xi)(\Univ v)}$ becomes an isomorphism of tangent categories. Thus, $\Xi(\X,\TT)$ admits the construction of vector fields and $\Xi(\VF(\X,\TT))$ is the tangentad of vector fields of $\Xi(\X,\TT)$.
\end{proof}

Proposition~\ref{proposition:equivalence-vector-fields} together with the Grothendieck $2$-equivalence $\TngFib\cong\TngIndx$ allows us to compute the vector fields of tangent indexed categories. Consider a tangent indexed category $(\X,\TT;\IND,\TT')$ whose underlying indexed category $\IND\colon\X^\op\to\Cat$ sends an object $M$ of $\X$ to the category $\X^M$ and a morphism $f\colon M\to N$ to the functor $f^\*\colon\X^N\to\X^M$. Define $\VF(\X,\TT;\IND,\TT')$ as follows:
\begin{description}
\item[Base tangent category] The base tangent category is $\VF(\X,\TT)$;

\item[Indexed category] The indexed category $\VF(\IND)\colon\VF(\X,\TT)^\op\to\Cat$ sends a vector field $(M,v)$ of $(\X,\TT)$ to the category $\VF^{(M,v)}$ whose objects are pairs $(E,w)$ formed by an object $E\in\X^M$ and a morphism $w\colon E\to v^\*\T'^ME$ which satisfies the following condition
\begin{equation*}
\begin{tikzcd}
& {v^\*\T'^ME} \\
E & {(p\o v)^\*E} & {v^\*p^\*E}
\arrow["{v^\*p'^M}", from=1-2, to=2-3]
\arrow["w", from=2-1, to=1-2]
\arrow["{\IND_1}"', from=2-1, to=2-2]
\arrow["{\IND_2}"', from=2-2, to=2-3]
\end{tikzcd}
\end{equation*}
where $p\o v=\id_M$, $\IND_1\colon E\cong\id_E^\*E$ is the unitor of $\IND$, and $\IND_2\colon(f\o g)^\*E\cong g^\*f^\*E$ is the compositor of $\IND$. A morphism $\varphi\colon(E,w)\to(E',w')$ of $\VF^{(M,v)}$ is a morphism $\varphi E\to E'$ such that the following diagram commutes:
\begin{equation*}
\begin{tikzcd}
{v^\*\T'^ME} && {v^\*\T'^ME'} \\
E && {E'}
\arrow["{v^\*\T'^M\varphi}", from=1-1, to=1-3]
\arrow["w", from=2-1, to=1-1]
\arrow["\varphi"', from=2-1, to=2-3]
\arrow["{w'}"', from=2-3, to=1-3]
\end{tikzcd}
\end{equation*}
Furthermore, $\VF(\IND)$ sends a morphism $f\colon(N,u)\to(M,v)$ of vector fields to the functor
\begin{align*}
f^\*&\colon\VF^{(M,v)}\to\VF^{(N,u)}
\end{align*}
which sends a pair $(E,w)$ to the pair $(f^\*E,w_{f^\*})$, where $w_{f^\*}$ is the morphism so defined:
\begin{align*}
&w_{f^\*}\colon f^\*E\xrightarrow{f^\*w}f^\*v^\*\T'^ME=u^\*(\T f)^\*\T'^ME\xrightarrow{u^\*\xi^f}u^\*\T'^Nf^\*E
\end{align*}

\item[Indexed tangent bundle functor] The indexed tangent bundle functor $\VF(\T')$ consists of the list of functors
\begin{align*}
&\T'^{(M,v)}\colon\VF^{(M,v)}\to\VF^{(\T M,v_\T)}
\end{align*}
which send a pair $(E,w)\in\VF^{(M,v)}$ to the pair $(\T'^ME,w_{\T'^M})$ where $w_{\T'^M})$ is defined as follows:
\begin{align*}
&w_{\T'^M}\colon\T'^ME\xrightarrow{\T'^Mw}\T'^Mv^\*\T'^ME\xrightarrow{\xi^v}(\T v)^\*\T'^{\T M}\T'^ME\xrightarrow{(\T v)^\*c'^M}(\T v)^\*c^\*\T'^{\T M}\T'^ME\xrightarrow{\IND_2}(v_\T)^\*\T'^{\T M}\T'^ME
\end{align*}
Moreover, $\T'^{(M,v)}$ sends a morphism $\varphi\colon(E,w)\to(E',w')$ to $\T'^M\varphi$. The distributors of $\T'^{(M,v)}$ coincide with the distributors $\xi^f$ of $\T'^M$;

\item[Indexed natural transformations] The structural indexed natural transformations of $\VF(\IND)$ are the same as for $\IND$.
\end{description}

\begin{theorem}
\label{theorem:vector-fields-tangent-indexed-categories}
The $2$-category $\Indx$ of indexed categories admits the construction of vector fields. In particular, given a tangent indexed category $(\X,\TT;\IND,\TT')$, the tangent indexed category of vector fields of $(\X,\TT;\IND,\TT')$ is the tangent indexed category $\VF(\X,\TT;\IND,\TT')=(\VF(\X,\TT),\VF(\IND),\VF(\TT'))$.
\end{theorem}
\begin{proof}
Thanks to Proposition~\ref{proposition:equivalence-vector-fields}, the Grothendieck $2$-equivalence $\Fib\simeq\Indx$ of~\cite[Theorem~5.5]{lanfranchi:grothendieck-tangent-cats} implies that $\Indx$ admits the construction of vector fields. Furthermore, $\VF(\X,\TT;\IND,\TT')$ corresponds via the Grothendieck equivalence to $\VF(\Pi_{(\IND,\TT')})$, where $\Pi_{(\IND,\TT')}$ is the tangent fibration associated to $(\IND,\TT')$. By spelling out the details, one finds out that this corresponds to the tangent indexed category $(\VF(\X,\TT),\VF(\IND),\VF(\TT'))$.
\end{proof}

\subsection{Tangent split restriction categories}
\label{subsubsection:vector-fields-tangent-split-restriction-categories}
In this section, we compute the construction of vector fields for tangent split restriction categories.
\par Consider a tangent split restriction category $(\X,\TT)$ and let $\VF(\X,\TT)$ denote the category of pairs $(M,v)$ formed by an object $M$ of $(\X,\TT)$ and a total map $v\colon M\to\T M$, namely, $\bar v=\id_M$, and whose morphisms $f\colon(M,v)\to(N,u)$ are morphisms $f\colon M\to N$ of $(\X,\TT)$ which strictly commute with $v$ and $u$, namely, $u\o f=\T f\o v$ and whose restriction idempotent $\bar f\colon M\to M$ commutes with $v$, namely, $v\o\bar f=\T\bar f\o v$. The tangent split restriction structure on $(\X,\TT)$ lifts to $\VF(\X,\TT)$. In particular, the tangent bundle functor sends each pair $(M,v)$ to $(\T M,v_\T)$, where:
\begin{align*}
&v_\T\colon\T M\xrightarrow{\T v}\T^2M\xrightarrow{c}\T^2M
\end{align*}

\begin{theorem}
\label{theorem:vector-fields-tangent-split-restriction-categories}
The $2$-category $\sRestrCat$ of split restriction categories admits the construction of vector fields. In particular, the tangentad of vector fields of a tangent split restriction category $(\X,\TT)$ is the tangent split restriction category $\VF(\X,\TT)$.
\end{theorem}
\begin{proof}
By direct inspection, one can show that the $2$-category $\sRestrCat$ admits inserters and equifiers. Therefore, by Theorem~\ref{theorem:PIE-limits-vector-fields}, $\sRestrCat$ admits the construction of vector fields. By unwrapping this construction, one finds out that the tangentad of vector fields of a tangent split restriction category $(\X,\TT)$ is precisely $\VF(\X,\TT)$.
\end{proof}

\subsection{Tangent restriction categories: a general approach}
\label{subsubsection:vector-fields-tangent-restriction-categories}
General tangent restriction categories are not an example of tangentads. However, thanks to Proposition~\cite[Lemma~4.36]{lanfranchi:tangentads-I}, every tangent restriction category embeds into a tangent split restriction category, which, by~\cite[Proposition~4.35]{lanfranchi:tangentads-I}, is a tangentad in the $2$-category $\sRestrCat$ of split restriction categories.
\par In this section, we introduce a general procedure for extending the formal theory of vector fields to tangent-like concepts which fail to be tangentads. Thus, we apply this general construction to the case of tangent restriction categories.

\begin{definition}
\label{definition:extension-setting}
A \textbf{pullback-extension context} consists of two $2$-categories $\CC$, and $\DD$, two $2$-functors
\begin{align*}
\Xi&\colon\DD\leftrightarrows\Tng(\CC)\colon\Inc
\end{align*}
and a natural $2$-transformation
\begin{align*}
\eta_\X&\colon\X\to\Inc(\Xi(\X))
\end{align*}
for $\X\in\DD$.
\end{definition}

Consider the functors
\begin{align*}
&\Inc\colon\Tng(\sRestrCat)\cong\TngsRestrCat\to\TngRestrCat\\
&\Xi=\Split_R\colon\TngRestrCat\to\TngsRestrCat\cong\Tng(\sRestrCat)
\end{align*}
where $\Inc$ is the inclusion of $\TngsRestrCat$ into $\TngRestrCat$ as a sub $2$-category and $\Split_R$ sends each tangent restriction category $(\X,\TT)$ into the tangent split restriction category $(\bar\X,\bar\TT)$ obtained by formally splitting the restriction idempotents of $(\X,\TT)$. Thus, $\Inc$ and $\Xi$ together with the natural $2$-transformation
\begin{align*}
&\eta_{(\X,\TT)}\colon(\X,\TT)\to(\bar\X,\bar\TT)=\Inc(\Split_R(\X,\TT))
\end{align*}
which embeds $(\x,\TT)$ to its splitting completion, form a pullback-extension context. In particular, the objects of $(\bar\X,\bar\TT)$ are pairs $(A,e)$ formed by an object $A$ of $\X$ and a restriction idempotent $e=\bar e$ of $\X$. Thus, $\varphi$ sends each $A\in(\X,\TT)$ to $(A,\id_A)$.
\par Consider a pullback-extension context $(\CC,\DD;\Xi,\Inc,\Xi;\eta)$. Assume also that $\CC$ admits the construction of vector fields. We want to extend the construction of vector fields on $\DD$, by pulling back along $\eta$. In particular, let us assume that the following $2$-pullback diagram exists in $\DD$:
\begin{equation}
\label{equation:pullback-extension-vector-fields}
\begin{tikzcd}
{\VF(\X)} & {\Inc(\VF(\Xi(\X))} \\
\X & {\Inc(\Xi(\X))}
\arrow["{\VF(\eta)}", from=1-1, to=1-2]
\arrow["{\U_\X}"', from=1-1, to=2-1]
\arrow["\lrcorner"{anchor=center, pos=0.125}, draw=none, from=1-1, to=2-2]
\arrow["{\Inc(\U_{\Xi(\X)})}", from=1-2, to=2-2]
\arrow["\eta"', from=2-1, to=2-2]
\end{tikzcd}
\end{equation}

\begin{definition}
\label{definition:vector-fields-extension}
Let $(\CC,\DD;\Inc,\Xi;\eta)$ be a pullback-extension context. An object $\X$ of $\DD$ admits the \textbf{extended construction of vector fields} (w.r.t. to the pullback-extension context) if $\Xi(\X)$ admits the construction of vector fields in $\CC$ and the $2$-pullback diagram of Equation~\eqref{equation:pullback-extension-vector-fields} exists in $\DD$. In this scenario, the \textbf{object of vector fields} (w.r.t. to the pullback-extension context) of $\X$ is the object $\VF(\X)$ of $\DD$, namely, the $2$-pullback of $\Inc(\U_{\Xi(\X)})$ along $\eta$.
\end{definition}

We have already shown that the $2$-category $\TngRestrCat$ of tangent restriction categories admits a pullback-extension context, whose base $2$-category $\CC$ is the $2$-category $\RestrCat$ of restriction categories and the natural $2$-transformation $\eta$ is the inclusion of each tangent restriction category $(\X,\TT)$ into its formal split completion $\Inc(\Split_R(\X,\TT))$. The goal is to use this context to extend the construction of vector fields to general tangent (non-necessarily split) restriction categories. Let us start by guessing what vector fields of a tangent restriction category would look like.
\par In the previous section, we showed that the tangentad of vector fields of a tangent split restriction tangent category $(\X,\TT)$ is the tangent split restriction category $\VF(\X,\TT)$ whose objects are pairs $(M,v)$ formed by an object $M$ of $(\X,\TT)$ together with a total map $v\colon M\to\T M$ which is a section of the projection, namely, $p\o v=\id_M$ and morphisms $f\colon(M,v)\to(N,u)$ are morphisms $f\colon M\to N$ of $(\X,\TT)$ between the underlying objects of $(\X,\TT)$ which are compatible with $v$ and $u$, namely, $u\o f=\T f\o v$, and whose restriction idempotent is $\bar f$. We showed that each restriction idempotent splits in $\VF(\X,\TT)$.
\par Now, consider a generic tangent restriction category $(\X,\TT)$ and let us unwrap the definition of $\VF(\Split_R(\X,\TT))$. The objects of $\VF(\Split_R(\X,\TT))$ are pairs $(M,v)$ formed by an object $M$ of $(\X,\TT)$ together with a map $v\colon M\to\T M$ of $(\X,\TT)$, which satisfies the following condition:
\begin{align*}
&\T\bar v\o v=v
\end{align*}
A morphism $f\colon(M,v)\to(M',v')$ of $\VF(\Split_R(\X,\TT))$ consists of a morphism $f\colon M\to M'$ of $(\X,\TT)$ which commutes with the vector fields and the restriction idempotents, namely:
\begin{align*}
&v'\o f=\T f\o v\\
&\bar{v'}\o f=f=f\o\bar v
\end{align*}
The tangent bundle functor sends each $(M,v)$ to $(\T M,c\o\T v)$.
\par Define $\VF(\X,\TT)$ to be the full subcategory of $\VF(\Split_R(\X,\TT))$ spanned by the objects $(M,v)$ where $v\colon M\to\T M$ is total in $(\X,\TT)$, namely, $\bar v=\id_M$.

\begin{lemma}
\label{lemma:vector-fields-tangent-restriction-categories}
Let $(\X,\TT)$ be a tangent restriction category. The subcategory $\VF(\X,\TT)$ of $\VF(\Split_R(\X,\TT))$ spanned by the objects $(M,v)$ where $v$ is total in $(\X,\TT)$ is a tangent restriction category.
\end{lemma}
\begin{proof}
For starters, notice that $\VF(\X,\TT)$ inherits from $(\X,\TT)$ a restriction structure, where the restriction idempotent of $f\colon(M,v)\to(M',v')$ is $\bar f$. Consider $\Split_R(\VF(\X,\TT))$, whose objects are triples $(M,v;e)$ formed by an object $(M,v)$ of $\VF(\X,\TT)$ and an idempotent $e=\bar e$, satisfying $v\o e=\T e\o v$. A morphism $f\colon(M,v;e)\to(M',v';e')$ of $\Split_R(\VF(\X,\TT))$ is a morphism $f\colon(M,v)\to(M',v')$ of $\VF(\X,\TT)$, such that $e'\o f=f=f\o e$ and whose restriction idempotent is $\bar f$. Define the following functor
\begin{align*}
&\Split_R(\VF(\X,\TT))\to\VF(\Split_R(\X,\TT))
\end{align*}
which sends each $(M,v;e)$ to $(M,v\o e)$ and each morphism $f\colon(M,v;e)\to(M',v';e')$ to $f\colon(M,v\o e)\to(M',v'\o e')$. Thus, $\Split_R(\VF(\X,\TT))$ can be seen as a subcategory of $\VF(\Split_R(\X,\TT))$. By restricting the tangent split restriction category of $\VF(\Split_R(\X,\TT))$ to $\Split_R(\VF(\X,\TT))$ we obtain a tangent split restriction category. However, $(\X,\TT)$ is a tangent restriction category if and only if $\Split_R(\X,\TT)$ is a tangent split restriction category. This proves that $\VF(\X,\TT)$ is a tangent restriction category.
\end{proof}

We can finally prove the main theorem of this section.

\begin{theorem}
\label{theorem:vector-fields-tangent-restriction-categories}
Consider the pullback-extension context $(\RestrCat,\TngRestrCat;\Inc,\Split_R;\eta)$ of tangent restriction categories. The $2$-category $\TngRestrCat$ admits the extended construction of vector fields with respect to this pullback-extension context. Moreover, the object of vector fields of a tangent restriction category is the tangent restriction category $\VF(\X,\TT)$.
\end{theorem}
\begin{proof}
Consider the following diagram:
\begin{equation*}
\begin{tikzcd}
{\VF(\X,\TT)} & {\Inc(\VF(\Split_R(\X,\TT)))} \\
{(\X,\TT)} & {\Inc(\Split_R(\X,\TT))}
\arrow["{\VF(\eta)}", from=1-1, to=1-2]
\arrow["\U"', from=1-1, to=2-1]
\arrow["{\Inc(\U_{\Split_R})}", from=1-2, to=2-2]
\arrow["\eta"', from=2-1, to=2-2]
\end{tikzcd}
\end{equation*}
where $\VF(\eta)$ sends each pair $(M,v)$ to $(M,v)$ and each morphism $f\colon(M,v)\to(M',v')$ to $f$ and $\U$ forgets the vector field, namely, $\U(M,v)=M$. Consider also another tangent restriction category $(\X',\TT')$ and two tangent restriction functors making the following diagram commute:
\begin{equation*}
\begin{tikzcd}
{(\X',\TT')} \\
& {\VF(\X,\TT)} & {\Inc(\VF(\Split_R(\X,\TT)))} \\
& {(\X,\TT)} & {\Inc(\Split_R(\X,\TT))}
\arrow["{(G,\alpha)}", curve={height=-18pt}, from=1-1, to=2-3]
\arrow["{(F,\alpha)}"', curve={height=12pt}, from=1-1, to=3-2]
\arrow["{\VF(\eta)}", from=2-2, to=2-3]
\arrow["\U"', from=2-2, to=3-2]
\arrow["{\Inc(\U_{\Split_R})}", from=2-3, to=3-3]
\arrow["\eta"', from=3-2, to=3-3]
\end{tikzcd}
\end{equation*}
In particular, let $GA=(FA,v_{FA})$ for each $A\in\X'$. By functoriality, $G(\id_A)=\id_{(FA,v_{FA})}=\bar{v_{FA}}$, where we used that the identity morphisms in $\VF(\Split_R(\X,\TT))$ are the idempotents, namely, $\id_{(M,v)}=\bar v$. However, we can also compute:
\begin{align*}
&\id_{FA}=\eta(F(\id_A))=\Inc(\U)(G(\id_A))=\bar{v_FA}
\end{align*}
Therefore, the vector field $v_{FA}\colon FA\to\T FA$ must be total in $(\X,\TT)$. Thus, we can define a functor
\begin{align*}
(H,\gamma)&\colon(\X',\TT')\to\VF(\X,\TT)
\end{align*}
which sends each $A$ to $(FA,v_{FA})$. It is not hard to see that this functor extends to a tangent morphism and that it is the unique functor which makes the following diagram commute:
\begin{equation*}
\begin{tikzcd}
{(\X',\TT')} \\
& {\VF(\X,\TT)} & {\Inc(\VF(\Split_R(\X,\TT)))} \\
& {(\X,\TT)} & {\Inc(\Split_R(\X,\TT))}
\arrow["{(H,\gamma)}", dashed, from=1-1, to=2-2]
\arrow["{(G,\alpha)}", curve={height=-18pt}, from=1-1, to=2-3]
\arrow["{(F,\alpha)}"', curve={height=12pt}, from=1-1, to=3-2]
\arrow["{\VF(\eta)}", from=2-2, to=2-3]
\arrow["\U"', from=2-2, to=3-2]
\arrow["{\Inc(\U_{\Split_R})}", from=2-3, to=3-3]
\arrow["\eta"', from=3-2, to=3-3]
\end{tikzcd}
\end{equation*}
We leave the reader to complete the proof that $\VF(\X,\TT)$ is the $2$-pullback of $\Inc(\U_{\Split_R})$ along $\eta$.
\end{proof}




\begingroup
\begin{thebibliography}{}

\bibitem{bauer:infinity-tangent-cats}Bauer, K., Burke, M. \& Ching, M. Tangent infinity-categories and Goodwillie calculus. 2021. eprint: arxiv:2101.07819.

\bibitem{blackwell:2-monad-theory}Blackwell, R., Kelly, G. \& Power, A. Two-dimensional monad theory. In:{\em Journal Of Pure And Applied Algebra}. {59.1} (1989), pp. 1--41.

\bibitem{cockett:tangent-cats}Cockett, J. \& Cruttwell, G. Differential Structure, Tangent Structure, and SDG. In:{\em Applied Categorical Structures}. {22.2} (2014), pp. 331--417.

\bibitem{cockett:jacobi}Cockett, J. \& Cruttwell, G. The Jacobi identity for tangent categories. In:{\em Cahiers De Topologie Et Géométrie Différentielle Catégoriques}. {56} (2015), pp. 301--316.

\bibitem{cockett:connections}Cockett, J. \& Cruttwell, G. Connections in tangent categories. In:{\em Theory And Applications Of Categories}. {32.26} (2017), pp. 835--888.

\bibitem{cockett:differential-bundles}Cockett, J. \& Cruttwell, G. Differential Bundles and Fibrations for Tangent Categories. In:{\em Cahiers De Topologie Et Géométrie Différentielle Catégoriques}. {LIX} (2018), pp. 10--92.

\bibitem{cockett:differential-equations}Cockett, J., Cruttwell, G. \& Lemay, J. Differential equations in a tangent category i: Complete vector fields, flows, and exponentials. In:{\em Applied Categorical Structures}. (2021), pp. 1--53.

\bibitem{cockett:restrictionI}Cockett, J. \& Lack, S. Restriction categories. I. Categories of partial maps. In:{\em Theoret. Comput. Sci.}. {270.1-2} (2002), pp. 223--259.

\bibitem{cockett:restrictionIII}Cockett, J. \& Lack, S. Restriction categories. III. Colimits, partial limits and extensivity. In:{\em Math. Structures Comput. Sci.}. {17.4} (2007), pp. 775--817.

\bibitem{cruttwell:algebraic-geometry}Cruttwell, G. \& Lemay, J. Differential Bundles in Commutative Algebra and Algebraic Geometry. 2023. eprint: arXiv:2301.05542.

\bibitem{cockett:tangent-monads}Cockett, J., Lemay, J. \& Lucyshyn-Wright, R. Tangent Categories from the Coalgebras of Differential Categories. In:{\em 28th EACSL Annual Conference On Computer Science Logic (CSL 2020)}. {152} (2020),  pp. 17:1-17:17.

\bibitem{cruttwell:connections-algebraic-geometry}Cruttwell, G., Lemay, J. \& Vandenberg, E. A tangent category perspective on connections in algebraic geometry. In:{\em Appl. Categ. Structures}. {33.1}, 4 (2025), pp. 40.

\bibitem{ikonicoff:operadic-algebras-tagent-cats}Ikonicoff, S., Lanfranchi, M. \& Lemay, J. The Rosický tangent categories of algebras over an operad. In:{\em High. Struct.}. {8.2} (2024), pp. 332--385.

\bibitem{lack:limits-lax-morphisms}Lack, S. Limits for lax morphisms. In:{\em Applied Categorical Structures}. {13.3} (2005), pp. 189--203.

\bibitem{lanfranchi:tangentads-I}Lanfranchi, M. Tangentads: a formal approach to tangent categories. 2025. eprint: arxiv:2503.18354.

\bibitem{lanfranchi:grothendieck-tangent-cats}Lanfranchi, M. The Grothendieck construction in the context of tangent categories. In:{\em Mathematical Structures In Computer Science}. {35} (2025), pp. e19.

\bibitem{leung:weil-algebras}Leung, P. Classifying tangent structures using Weil algebras. In:{\em Theory And Applications Of Categories}. {32} (2017), pp. 286--337.

\bibitem{power:pie-limits}Power, J. \& Robinson, E. A characterization of pie limits. In:{\em Mathematical Proceedings Of The Cambridge Philosophical Society}. {110}, 1 (1991), pp. 33--47.

\bibitem{rosicky:tangent-cats}Rosický, J. Abstract Tangent Functors. In:{\em Diagrammes}. {12} (1984), pp. JR1--JR11.

\bibitem{street:fibrations-yoneda-lemma}Street, R. Fibrations and Yoneda's lemma in a 2-category. In:{\em Category Seminar (Proc. Sem., Sydney, 1972/1973)}. {420} (1974), pp. 104--133.


\end{thebibliography}
\endgroup
\end{document}